\documentclass{article}
\usepackage{amsmath,amssymb,amsfonts,amsthm}
\usepackage{bm,graphicx,comment,color}
\usepackage{algorithmic,algorithm}
\usepackage{mathtools}
\usepackage{mathrsfs}
\usepackage{amsthm}

\usepackage{comment}
\usepackage{subfigure} 
\usepackage{ulem}
\usepackage{cancel}

\usepackage{multirow}
\usepackage{longtable}
\usepackage{multirow}
\usepackage{makecell}
\usepackage{xcolor}

\usepackage{ragged2e}\justifying\let\raggedright\justifying

\usepackage{graphicx,url,pgf}
\usepackage{color}
\usepackage{booktabs}
\usepackage[all]{xy}
\usepackage{latexsym}
\usepackage{graphics}
\usepackage{tabularx}
\usepackage{mathrsfs}
\usepackage{delarray}
\usepackage{subfig}
\usepackage{comment}

\numberwithin{equation}{section}
\newtheorem{theorem}{Theorem}[section]

\newtheorem{lemma}[theorem]{Lemma}

\newtheorem{remark}[theorem]{Remark}
\newtheorem{example}[theorem]{Example}

\newcommand{\W}{\mathcal W}
\newcommand{\w}{_{_{\mathcal W}}}

\newcommand{\EQT}{\mathcal E \mathcal Q \mathcal T}

\newcommand{\one}{{\bf 1}}

\usepackage[numbers,sort&compress]{natbib}

\begin{document}

\title{Theoretical analysis and numerical solution {of} a nonlinear vector equation}
\author{Yuezhi Wang\quad Gwi Soo Kim \quad Jie Meng}
\maketitle

\begin{abstract}
Theoretical and computational properties of a  vector equation 
$Ax-\|x\|_1x=b$  {are}  investigated, where {$A$} is an invertible $M$-matrix and $b$ is a nontrivial nonnegative vector. Specifically,  A is a structured $M$-matrix of the form $A=\alpha I_n-T_n(a)$
 with $T_n(a)$  being a Toeplitz matrix and $\alpha$ a specific scalar. Existence and uniqueness of a nonnegative solution is proved from a functional point of view, followed by the computation of the  solution and the perturbation analysis of the   equation. In particular,  
 a structure-preserving doubling algorithm (SDA) is analyzed and proved using a technique distinct from that for the Riccati equations. We show that SDA always applies and the convergence rate is at least 1/2. Numerical experiments are performed to demonstrate the effectiveness of the proposed algorithms. 
 
 \end{abstract}
 
\section{Introduction}
We are concerned with a vector equation
\begin{equation}\label{eq:equation1}
Ax-\|x\|_1x=b,
\end{equation}
where  $A\in \mathbb R^{n\times n}$ is an invertible $M$-matrix and $b\in \mathbb R^{n}$ is a nontrivial nonnegative vector. More precisely, the matrix $A$ can be written as $A=\alpha I_n-T_n(a)\in \mathbb R^{n\times n}$, where $T_n(a)$ is a Toeplitz matrix associated with $a(z)=\sum_{i=1-n}^{n-1}a_iz^i$ and $\alpha =2-\|a\|_{\w} $, where $\|a\|_{\w}:=\sum_{i=1-n}^{n-1}|a_i|$. Moreover, the vector $b$ satisfies $\|b\|_1<(1-\|a\|_{\w})^2$. In the infinite dimensional case,  this equation is  encountered in the numerical treatment  of the square root of an infinite $M$-matrix, which has the form $G+\one u^T$, where $G$ is a quasi-Toeplitz $M$-matrix and $u$ is an infinite vector with finite 1-norm. We refer the reader to \cite{arxiv_jie} for more details.

Quasi-Toeplitz matrices which can be written as the sum of a Toeplitz matrix and a low-rank correction matrix are ubiquitous in  Markov processes with infinite levels and phases \cite{motyer-taylor}. Recently, semi-infinite quasi-Toeplitz matrices of the form $A=G+{\bf 1}v_A^T$, where $G$ is a semi-infinite quasi-Toeplitz matrix and $v_A\in \ell^1$,   have been treated in some recent papers, see for instance \cite{BSM1,BMML2020,BSR,BMM,kim_meng}, due to their various applications such as  in the analysis of two-dimensional random walks in the quarter plane. One  subclass of the quasi-Toeplitz matrices is the $M$-matrices with a quasi-Toeplitz structure, which have been investigated in \cite{ckm,laa_jie,arxiv_jie}  in terms of their theoretical and computational properties.

  An semi-infinite quasi-Toeplitz $M$-matrix which is of interest usually has the form  $W=T_W+E_W+\one v_W^T$,   where $T_W$ is a semi-infinite Toeplitz matrix, $E_W$ is a low-rank matrix, and $v_W\in \ell^1$. It has been proved in \cite{arxiv_jie} that under mild conditions, matrix $W$ admits a unique square root $G=T_G+E_G+\one v_G^T $  that is also an semi-infinite quasi-Toeplitz $M$-matrix. In \cite{eqtexample}, invertible $M$-matrices of the form $W=T(w)+\one v_W^T$ are encountered, where $T(w)$ is an upper triangular quasi-Toeplitz $M$-matrix  associated with $w(z)=\sum_{i=0}^{\infty}w_iz^i$ and $v_W\in \ell^1$. Then in  \cite{arxiv_jie}, it was shown that matrix  $W$ admits a square root $G=T(g)+\one v_G^T$, where  $v_G$ solves an equation of the form \eqref{eq:equation1} with  semi-infinite coefficients. 
 In this paper, we show that the vector $v_G$ can be approximated by a vector  $\widehat{v}_G=(v^T,0)^T$,  where $v\in \mathbb R^p$ for sufficiently large $p$, and solves an equation that can be formalized as equation \eqref{eq:equation1}. This way, the computation of the vector $\widehat{v}_G$ is transformed to solve a finite-dimensional vector equation \eqref{eq:equation1}.

We propose a framework for the numerical treatment of the  finite-dimensional vector equation \eqref{eq:equation1}. In particular, we observe that equation \eqref{eq:equation1} has  a solution $x=(A-\mu_xI)^{-1}b$, where the matrix $A$ is a finite Toeplitz $M$-matrix,  $\mu_x$  satisfies  $\mu_x=\|(A-\mu_x I)^{-1}b\|_1$, from which we deduce a function $f(\mu)=\|(A-\mu  I)^{-1}b\|_1$ and  $\mu_x$ is a fixed point of $f(\mu )$.  The existence and uniqueness of a solution to equation \eqref{eq:equation1} can be obtained in terms of the existence of a fixed point of the function {$f(\mu)$}.  On the other hand, when the coefficient matrices of equation \eqref{eq:equation1} are of infinite size, two fixed-point iterations are proposed for the numerical treatment of the required solution \cite{arxiv_jie}. Instead of adjusting the fixed-point iterations to the finite-dimensional case, we propose new iterations, including a relaxed fixed-point iteration and Newton's iteration, to find a fixed point $\mu_x$ of  $f(\mu)$, followed by setting the solution of equation \eqref{eq:equation1} to $x=(A-\mu_xI)^{-1}b$.  Moreover, we show that a structure-preserving doubling algorithm can be applied, and the convergence is at least linear with rate 1/2.  Finally, we provide an upper perturbation bound for the solution.

This paper is organized as follows.  In the rest of this section, we recall some definitions and key properties related to semi-infinite quasi-Toeplitz $M$-matrices. In Section 2 we show how equation \eqref{eq:equation1} is obtained, followed by a proof of the existence and uniqueness of the solution of interest. In Section 3, we propose algorithms designed to compute the required solution, and in Section 4, we give an upper perturbation bound of the solution, concerning perturbations in the coefficient matrix $A$ and the vector $b$. Finally, in Section 5, we demonstrate the effectiveness of the proposed algorithms through numerical experiments. 

\subsection{Preliminary Concepts}
Let $a(z)=\sum_{i\in \mathbb Z}a_iz^i$  be a complex valued  function on the unit circle $\mathbb T$, a semi-infinite Toeplitz matrix $T(a)$ is said to be associated with a symbol $a(z)$   if  $(T(a))_{i,j}=a_{j-i}$. In the infinite case,  $a(z)$ belongs to the Wiener algebra $\W$, which is defined as $\mathcal W:=\{a(z): z\in \mathbb T,  \|a\|_{\w }:=\sum_{i\in \mathbb Z}|a_i|<\infty\}$, so that with $\|T(a)\|_{\infty}=\|a\|_{\w }$  the operator $T(a)$ is a bounded linear operator from $\ell^{\infty}$ to itself \cite{AB}. 
For an $n\times n$ Toeplitz matrix $T_n(a)$, we say that $T_n(a)$ is associated with a symbol  $a(z)=\sum_{k=1-n}^{n-1}a_kz^k$ in the sense that   $(T(a))_{i,j}=a_{j-i}$ for $1\leq i,j\leq n$. 
In this paper, we consider Toeplitz matrices with real elements, and when referring to $T(a) $ and $T_n(a)$, we assume the function $a(z)$ has only real coefficients.

Finite $M$-matrices are a subset of the real square matrices whose off-diagonal elements are nonpositive. In particular,   according to \cite[Definition 6.17]{higham}, 
an $n\times n$ real matrix $A$ is an $M$-matrix  if 
\begin{align*}
A=sI -B,\quad {\rm with }\ B\geq 0\quad {\rm and}\ s\geq \rho(B),
\end{align*}
where $\rho(B)$ is the spectral radius of $B$. If $s>\rho(B)$, then $A$ is an invertible $M$-matrix.

  Semi-infinite quasi-Toeplitz $M$-matrices are a subset of the semi-infinite quasi-Toeplitz matrices, the latter  can be written as
 \begin{equation}\label{eqt}
{{A=T(a)+E+\one v^T,}} 
\end{equation}
where ${{a(z)}} \in \mathcal W$, $E=(e_{i,j})_{i,j\in \mathbb Z^+}$ is such that $\lim_{i\rightarrow \infty}\sum_{j=1}^{\infty}|e_{i,j}|=0$, and $v=(v_1,v_2,\ldots)^T\in \ell^1$. This way, the matrix ${{A}}$ is a bounded linear operator from $\ell^{\infty}$ to itself \cite{BMML2020}.  Matrices with format \eqref{eqt} form a set $\EQT_{\infty}$,  which has been proved to be a Banach algebra \cite{BMML2020}. 
On the other hand, it has been deduced in \cite{arxiv_jie} that the characterization  of quasi-Toeplitz $M$-matrices is formally identical to finite $M$-matrices, that is, 
 a semi-infinite quasi-Toeplitz matrix ${{A}} \in \EQT_{\infty}$ is an infinite $M$-matrix   if it  can be expressed as ${{A=\beta I-A_1}} $  with  ${{A_1}}\geq 0$ and $\rho({{A_1}} )\leq \beta$. Moreover, if $\rho({{A_1}} )<\beta$, then ${{A}} $ is an invertible quasi-Toeplitz $M$-matrix.   We mention that for general infinite $M$-matrices and $M$-operators in a real Banach space,  the characterization is typically more complicated  and may depend more on the specific structure of the underlying space. We  refer the reader to \cite{AS,IM3,PN} for more details of definitions and properties of $M$-operators in a real Banach space.

Semi-infinite quasi-Toeplitz matrices and $M$-matrices possess many elegant properties due to the typical structure. We recall some of the properties that will be used in the subsequent analysis. 

\begin{lemma}\cite[Theorem 1.14]{AB}\label{lem:norm}
    If $a\in \mathcal W$, then $\|T(a)\|_1=\|T(a)\|_{\infty}=\|a\|_{\w}$. 
\end{lemma}
For any $n\in \mathbb Z^+$, we denote by $T_n(a)$ the principal $n\times n$ section of $T(a)$.   In view of Lemma \ref{lem:norm}, we have $\|T_n(a)\|_1\le \|T(a)\|_1= \|a\|_{\w}$. Moreover, if $A:=sI-T(a)$ is an invertible $M$-matrix, then according to \cite[Theorem 2.13]{laa_jie}, we know that $A_n:=sI_n-T_n(a)$ is also an invertible $M$-matrix. This way, we obtain $A_n^{-1}\geq 0$ and if $\|a\|_{\w}<s$, we have 
\begin{align}\label{norm1}
\|A_n^{-1}\|_1\le \frac1s\sum_{i=0}^{\infty}\Big(\frac{\|T_n(a)\|_{1}}{s}\Big)^i
\le \frac1s\sum_{i=0}^{\infty}\Big(\frac{\|a\|_{\w}}{s}\Big)^i
=\frac{1}{s-\|a\|_{\w}}.
\end{align}

\begin{lemma}\cite[Corollary 1.11]{AB}\label{AB}
$T(a)$ with $a(z)\in \W$ is invertible on $\ell^{\infty}$ if and only if   $a(z)\ne 0$ for all $z\in \mathbb T$ and  ${\rm wind}\ a=0$, where  ${\rm wind}\ a$ means the winding number of the function $a(z)$. 
\end{lemma}

\begin{lemma}\cite[Lemma 2.18]{BMML2020}\label{lem:eqt}
	Suppose $A=T(a)+E_A+\one v_A^T\in \EQT_{\infty}$ and  $B=T(b)+E_B+\one v_B^T\in \EQT_{\infty}$, then $C=AB\in \EQT_{\infty}$ and $C=T(c)+E_C+\one v_C^T$, where $c=ab$ and $v_C=(\sum_{j\in \mathbb Z}a_j)v_B+B^Tv_A$. 
\end{lemma}

The following result can be deduced from Theorem 3.7, Lemma 2.3 and Lemma 3.1 in \cite{arxiv_jie}.

\begin{lemma}\label{lem:square}
Suppose  ${{A=I-(T(a)+ E_A+\one v_A^T)}}  \in \EQT_{\infty}$ satisfies ${{T(a)+E_A+\one v_A^T}}\geq 0$ and ${{\|T(a)+E_A+\one v_A^T\|_{\infty}}}<1 $, then 
\begin{itemize}
	\item[ {\rm i)}] ${{\|T(a)\|_{\infty}<1, \|T(a)+\one v_A^T\|_{\infty}=\|a\|_{\w}+\|v_A\|_1<1}}$;
	
	\item [{\rm ii)}]there is  a unique $G=T(g)+E_G+\one v_G^T\in \EQT_{\infty}$ such that $G\geq 0$, $\|G\|_{\infty}<1 $ and  $(I-G)^2={{A}}$;
	
	\item [{\rm iii)}] $T(g)\geq 0$, $v_G\geq 0$, and  $\|T(g)+\one v_G^T\|_{\infty}=\|g\|_{\w}+\|v_G\|_1<1 $;
	\item [{\rm iv)}]  The matrices $I-T(g)$ and $ I-T(g)-\one v_G^T$ are   both invertible quasi-Toeplitz $M$-matrices. 
\end{itemize}
\end{lemma}

The above lemma gives a bound to $\|g\|_{\w}$, that is, $\|g\|_{\w}<1-\|v_G\|_1$. We show that $\|g\|_{\w}\leq 1-\max \{\|v_A\|_1, \|v_G\|_1\}$. 

\begin{lemma}\label{lem:pro}
Suppose ${{A=I-(T(a)+E_A+\one v_A^T)}} \in \EQT_{\infty}$ and $G=I-(T(g)+E_G+\one v_G^T)$ is such that $G^2={{A}} $, then under the conditions in Lemma \ref{lem:square}, we have $\|g\|_{\w}+\|v_A\|_1<1$ and $\|v_A\|_1\leq (1-\|g\|_{\w})^2$.
\end{lemma}
\begin{proof}
In view of Lemma \ref{lem:square} and Lemma \ref{lem:eqt}, we know that  $\|v_A\|_1+\|a\|_{\w}<1$ and  $a(z)=2g(z)-g(z)^2$, from the second equation,  we obtain  $\|a\|_{\w}=\|2g-g^2\|_{\w}\geq 2\|g\|_{\w}-\|g\|_{\w}^2$, it follows that 
\[
\|v_A\|_1<1-\|a\|_{\w}<1-2\|g\|_{\w}+\|g\|_{\w}^2=(1-\|g\|_{\w})^2.
\]

Since $0<1-\|g\|_{\w}<1$, we have $(1-\|g\|_{\w})^2<1-\|g\|_{\w}$, so that $\|v_A\|_1\leq 1-\|g\|_{\w}$, that is, $\|v_A\|_1+\|g\|_{\w}< 1$. 

\end{proof}

\section{Theoretical analysis }\label{theo}
In this section, we  first show in detail where equation \eqref{eq:equation1} is derived from, then we  investigate its solution of interest. In what follows, the dimension of the identity matrix 
$I$ is determined by context.
\subsection{Construction of the problem}

Under the conditions in Lemma \ref{lem:square}, the matrix $A=I-T(a)-E_A-\one v_A^T\in \EQT_{\infty}$ admits a unique square root of the form $G=I-T(g)-E_G-\one v_G^T$, that is,  $G^2=A$. For the case where $E_G$ is trivial,  the computation of $G$ can be divided into two parts, one of which is to compute the Toeplitz matrix $T(g)$ based on the evaluation/interpolation technique \cite{ckm}. This way, for a given $\epsilon>0$,  $T(g)$ can be approximated by $T(\tilde{g})$ with $\tilde{g}(z)=\sum_{i=-N}^N \tilde{g}_iz^i$  and it holds $\|g-\tilde{g}\|_{\w}<\epsilon$. 

Then it  remains  to compute the nonnegative vector $v_G$, which was shown  in  \cite[page 14]{arxiv_jie}   to satisfy
\begin{equation}\label{eq:sys1}
(\alpha I-T(g))^Tv_G-\mu_G v_G=v_A,
\end{equation}
where  $\alpha=2-\|g\|_{\w}$ and $\mu_G=\|v_G\|_1$.  Observe that  in the numerical treatment,   $T(g)$ is approximated by $T(\tilde{g})$, where $\tilde{g}(z)$  has only finitely many nonzero coefficients.  Then in the numerical treatment of equation \eqref{eq:sys1}, $T(g)$ is actually replaced by $T(\tilde{g})$.   Without loss of generality, we may assume $g(z)=\sum_{i=-(N-1)}^{N-1}g_iz^i$ for some constant $N$.

In real applications,  the vector $v_A$ always has a finite number of nonzero elements, see \cite{BMML2020} for instance. In the following, we assume that the vector $v_A$ contains only the first $s$ nonzero entries. On the other hand, since $v_G\in \ell^1$, that is, $\|v_G\|_1<\infty$, then for arbitrary small $\epsilon$, there is sufficiently large $t$, such that for $\widehat{v}_G=(\widehat{v}_1^T, 0)^T$   it holds 
\begin{equation}\label{eq:vg}  
\|\widehat{v}_G-v_G\|_1< \epsilon,
\end{equation}
where the vector $\widehat{v}_1$ is composed of  the first $t$ elements of vector $v_G$.  

Set $p=\max\{s, t, N\}$,  where $N$ is the bandwidth of Toeplitz matrix $T(g)$, and $s$ and $t$ are the number of nonzero elements of the vector ${v}_A$ and $\widehat{v}_G$, respectively.   We partition the matrix $(\alpha I-T(g))^T$ into $\left(\begin{array}{cc} {T}_{11}&{T}_{12} \\ {T}_{21}& {T}_{22}\end{array}\right)$, where ${T}_{11}=\alpha I_p-T_p(\bar{g})\in \mathbb R^{p\times p}$, ${T}_{12}\in \mathbb R^{p\times \infty}$, ${T}_{21}\in \mathbb R^{\infty\times p}$ and ${T}_{22}\in \mathbb R^{\infty\times \infty}$. 
Now  write $v_G=(v_1^T, v_2^T)^T$, where $v_1\in \mathbb R^p$ and $v_2\in \mathbb R^{\infty}$. Then it is clear that 
\begin{equation}\label{eq:normv2}
\|v_2\|_1\leq \|\widehat{v}_G-v_G\|_1<\epsilon,\ 
\|v_1\|_1\leq \|v_G\|_1<1,\ {\rm and}\ \|v_1\|+\|g\|_{\w}<1.  
\end{equation}
Moreover, 
we obtain from equation \eqref{eq:sys1} that 
\begin{equation}\label{eq:v1}
T_{11} v_1-\mu_G v_1=\widehat{v}_A-T_{12}v_2, 
\end{equation}
where $\widehat{v}_A \in \mathbb R^p$ is composed of the first $p$ elements of $v_A$, and it is clear that $\|\widehat{v}_A\|_1=\|v_A\|_1$.   
It can be seen that equation \eqref{eq:v1} is equivalent to 
\begin{equation}\label{eq:v11}
T_{11}v_1-\mu_1v_1=\widehat{v}_A-T_{12}v_2+(\mu_G-\mu_1)v_1,
\end{equation}
where $\mu_1=\|v_1\|_1$. 

Consider the equation 
\begin{equation}\label{eq:v2}
T_{11}y-\mu_y y=\widehat{v}_A,
\end{equation}
where $\mu_y=\|y\|_1$. If equation \eqref{eq:v2} has a solution $y$,  we show that the vector $v_1$ can be approximated by $y$. 

\begin{lemma}\label{lem:y}
Let $v_G=(v_1^T, v_2^T)^T$ with $v_1\in \mathbb R^p$ and $v_2\in \mathbb R^{\infty}$ be a solution of equation $\eqref{eq:sys1}$. For given $\epsilon>0$, if $\|v_2\|_1<\epsilon$ and  equation \eqref{eq:v2} has a solution $y$ such that $\|g\|_{\w }+\|y\|_1<1$, then  it satisfies $\|v_1-y\|_1<c\epsilon$ for some constant $c$. 
\end{lemma}

\begin{proof}
According to equations \eqref{eq:v1} and \eqref{eq:v2}, we have  
\[
T_{11}(v_1-y)-\mu_1 v_1+\mu_y y=-T_{12}v_2+(\mu_G-\mu_1)v_1,
\]
from which we obtain 
\[
(T_{11}-\mu_1 I)(v_1-y)= (\mu_1 -\mu_y)y-T_{12}v_2+(\mu_G-\mu_1)v_1.
\]

Since $T(\bar{g})=(g_{i-j})_{i,j\in \mathbb Z^+} \geq 0$, $\|T(\bar{g})\|_{\infty}=\|\bar{g}\|_{\w}=\|g\|_{\w}< 1$, and according to Lemma \ref{lem:square} we have $\|g\|_{\w}+\mu_G=\|g\|_{\w }+\|v_G\|_1 <1$, then it follows from \eqref{eq:normv2} that $\alpha -\mu_1\geq \alpha-\mu_G =2-(\|g\|_{\w }+\mu_G)>1$, so that  $(\alpha-\mu_1 ) I-T(\bar{g})$ is an invertible quasi-Toeplitz $M$-matrix. In view of \cite[Theorem 2.13]{laa_jie}, we know that $T_{11}-\mu_1 I$, which is the principal $p\times p$ section of $(\alpha-\mu_1 ) I-T(\bar{g})$,   is  an invertible $M$-matrix. Then, we obtain 
\begin{equation}\label{eq:cy}
v_1-y=(T_{11}-\mu_1I)^{-1}\big((\mu_1 -\mu_y)y-T_{12}v_2+(\mu_G-\mu_1)v_1\big).
\end{equation}
 Together with the fact that $|\mu_G-\mu_1|\leq \|v_G-v_1\|_1=\|v_2\|_1$,  it follows that  
\[
\|v_1-y\|_1\leq\|(T_{11}-\mu_1 I)^{-1}\|_{1}\big (  |\mu_1-\mu_y|\|y\|_1+(\|T_{12}\|_{1}+\|v_1\|_1)\|v_2\|_1\big ).
\]
According to \eqref{norm1}, we have $\|(T_{11}-\mu I)^{-1}\|_1\le \frac{1}{2-2\|g\|_{\w}-\mu_1}$, and it follows that 
\[
\|v_1-y\|_1\leq\frac{1}{2-2\|g\|_{\w}-\mu_1}\big (  |\mu_1-\mu_y|\|y\|_1+(\|T_{12}\|_{1}+\|v_1\|_1)\|v_2\|_1\big ).
\]
Observe that $\|T_{12}\|_{1} \leq \|\alpha I-T(\bar{g})\|_{1}\leq \alpha +\|g\|_{\w}{\le} 2$,  $|\mu_1-\mu_y|\le \|v_1-y\|_1$, and it follows from \eqref{eq:normv2} that $\|v_1\|_1<1-\|g\|_{\w}$.  We thus have 
\[
\|v_1-y\|_1\leq \frac{ \|y\|_1 }{2-2\|g\|_{\w}-\mu_1}\|v_1-y\|_1+\frac{ 3-\|g\|_{\w} }{2-2\|g\|_{\w}-\mu_1}\epsilon. 
\]
Recall that  $\|g\|_{\w}+\mu_1<1$ and we have assumed $\|g\|_{\w}+\|y\|_1<1$, from which we deduce that  
$
\frac{ \|y\|_1 }{2-2\|g\|_{\w}-\mu_1}<1.
$
Set $c=\frac{3-\|g\|_{\w}}{2-2\|g\|_{\w}-\mu_1-\|y\|_1},$ then $c>0$ and we have $\|v_1-y_1\|_1\leq c\epsilon$. The proof is complete.  
\end{proof}

Together with \eqref{eq:vg} and 
Lemma \ref{lem:y}, one can see that for sufficiently large $p$, the vector  $v_y=(y^T,0^T)^T$ with $y\in \mathbb R^p$ being the solution of equation \eqref{eq:v2} satisfies 
$\|v_G-v_y\|_1\leq (1+c)\epsilon$. This provides some insights on the  numerical treatment of the vector $v_G$, that is, if the constant $c$ is small,  we may compute the solution of equation \eqref{eq:v2}, followed by extending the solution to an infinite vector.

Concerning equation \eqref{eq:v2},  observe  that  $T_{11}=\alpha I_p -T_p(\bar{g})$ is an invertible $M$-matrix.  
On the other hand, $\hat{v}_A\geq 0$ and $\|\hat{v}_A\|_1=\|v_A\|_1<(1-\|g\|_{\w})^2$.  Replace $T_{11}$ and $\hat{v}_A$  in \eqref{eq:v2} by $A=\alpha I-T_n(a)$ and a nontrivial nonnegative vector  $b$, respectively, 
 it can be seen that equation \eqref{eq:v2}  is exactly equation \eqref{eq:equation1} if  $A$ and  $b$ satisfy 
\begin{equation}\label{eq:ba}
T_n(a)\geq 0, \quad \|a\|_{\w}+\|b\|_1<1\quad {\rm and}\quad \|b\|_1<(1-\|a\|_{\w})^2.
\end{equation}

 In what follows, we investigate the solvability of equation \eqref{eq:equation1}, and propose numerical algorithms for the numerical treatment of the solution $x$ such that $\|a\|_{\w}+\|x\|_1<1$. Indeed, according to Lemma \ref{lem:y},  a solution satisfying $\|a\|_{\w} + \|x\|_1 < 1$ is the one of interest. We denote this solution by $x_*$ in the remainder of this paper and show that such a solution is unique.

\subsection{Solvability of equation \eqref{eq:equation1}}

  In this subsection, we show that equation \eqref{eq:equation1} admits a unique solution $x_*$. 
 
\begin{lemma}\label{lem:mu0}
Equation \eqref{eq:equation1} has a solution $x$ if and only if   equation \begin{equation}\label{eq:mumu}
    \mu=\|(A-\mu I)^{-1}b\|_1
\end{equation} has a solution $\mu$. If  $\mu$ is a solution of equation \eqref{eq:mumu}, then $x=(A-\mu I)^{-1}b$ is a solution of equation \eqref{eq:equation1}. 
\end{lemma}

\begin{proof}
Suppose $x$ is a solution of equation \eqref{eq:equation1}, then we have $x=(A-\mu_xI)^{-1}b$, and it is clear that  $\mu_x=\|(A-\mu_x I)^{-1}b\|_1$, that is, $\mu_x$ is a solution of equation \eqref{eq:mumu}. On the other hand, suppose $\mu $ solves equation \eqref{eq:mumu}, set $x=(A-\mu I)^{-1}b$, then $\mu_x=\|(A-\mu I)^{-1}b\|_1=\mu$, and  
\[
Ax-\mu_x x=(A-\mu I)(A-\mu I)^{-1}b=b,
\]
that is, $x=(A-\mu I)^{-1}b$ is a solution of equation \eqref{eq:equation1}. 
\end{proof}

In view of Lemma \ref{lem:mu0}, one can see that the solvability of equation $\mu =\|(A-\mu I)^{-1}b\|_1$ implies the solvability of equation \eqref{eq:equation1}. Definite the function $f:[0,1-\|a\|_{\w}]\rightarrow \mathbb R$ such that 
 \begin{equation}\label{eq:f}
 f(\mu)=\|(A-\mu I)^{-1}b\|_1.
 \end{equation}
 For any $0\le \mu \le 1-\|a\|_{\w}$, we  get  $2-\|a\|_{\w} -\mu \geq 1> \|a\|_{\w}$. It follows that  $A-\mu I = (2-\|a\|_{\w}-\mu )I - T_n(a)$ is an invertible $M$-matrix, so that the function $f(\mu)$ is well defined. 

The function $f$ defined in \eqref{eq:f} plays a crucial role in the analysis of the existence and uniqueness of the solution to equation \eqref{eq:equation1}. We first show that it has  a unique fixed point in the interval $[0,1-\|a\|_{\w}]$. 

Since $A-\mu I$ is an invertible $M$-matrix for any $\mu\in [0,1-\|a\|_{\w}]$, so  $(A-\mu I)^{-1}\geq 0$. Together with the fact that $b\ge 0$, we have for any $\mu, \mu+\Delta \mu\in [0,1-\|a\|_{\w}]$ that 
\[
\|(A-(\mu+\Delta \mu)I)^{-1}b\|_1={\one}^T(A-(\mu+\Delta \mu)I)^{-1}b,
\]
and 
\[
\|(A-\mu I)^{-1}b\|_1={\one^T}(A-\mu I)^{-1}b.
\]
Then we have
\begin{align}\label{eq:deri0}
	f(\mu + \Delta \mu) - f(\mu) 
	&= {\one^T} \big((I - \Delta \mu (A-\mu I)^{-1})^{-1} -I\big)(A-\mu I)^{-1}b\notag \\
& =  {\one^T}(A-\mu I)^{-2}b\Delta \mu+o(\Delta \mu ),
\end{align}
where the last equality follows from the fact that
$$\big(I - \Delta \mu (A-\mu I)^{-1})^{-1}=\sum_{i=0}^{\infty}(\Delta \mu (A-\mu I)^{-1})^i, $$  for $|\Delta \mu |$ sufficiently small.  As a consequence of \eqref{eq:deri0}, we obtain 
 \begin{equation}\label{eq:deri}
 f'(\mu) = \|(A-\mu I)^{-2} b\|_1.
 \end{equation}

 With the help of function $f'(\mu)$, we are ready to show that equation \eqref{eq:equation1} has a unique solution $x_*$ such that $\|a\|_{\w}+\|x_*\|_1<1$, and we have 

\begin{theorem}\label{thm:solution}(Existence and uniqueness of the solution)
 {Among the solutions $x_*$ of equation  \eqref{eq:equation1} there is exactly one which satisfies} $\|a\|_{\w}+\|x_*\|_1<1$. 
\end{theorem}

\begin{proof}
According to Lemma \ref{lem:mu0}, it suffices to show there is a unique $\mu_*$ such that $\mu_* <1-\|a\|_{\w }$ and $\mu_*=f(\mu_*)$.   Observe that for any $\mu\in [0,1-\|a\|_{\w}]$, we have 
\begin{align}\label{ine:fu}
\|(A-\mu I)^{-2} b\|_1 
	&\le \|(A-(1-\|a\|_{\w}) I)^{-2} b\|_1\nonumber  \\
	&\le  \frac{\|b\|_1}{(1-\|a\|_{\w})^2}<1,
\end{align}
where the last inequality follows from \eqref{eq:ba}. Hence, $f'(\mu)<1$ for any $\mu \in [0,1-\|a\|_{\w }]$. Set $g(\mu)=f(\mu)-\mu$, we have \begin{equation}\label{gu}
    g'(\mu)=f'(\mu)-1 < 0,
\end{equation} that is, the function $g(\mu)$ is monotonically decreasing in the interval $[0,1-\|a\|_{\w}]$. Since
$g(0)=\|A^{-1}b\|_1>0,$
and 
\[
\begin{aligned}
g(1-\|a\|_{\w}) &= \|(I - T(a))^{-1}b\|_1 - 1+\|a\|_{\w}\\
& \le \frac{\|b\|_1}{1-\|a\|_{\w}}-1+\|a\|_{\w}\\
& =\frac{\|b\|_1-(1-\|a\|_{\w})^2}{1-\|a\|_{\w}}<0,
\end{aligned}
\]
 we know that the function $g(\mu)$ has one and only one root in the interval $[0,1-\|a\|_{\w}]$, that is, there is a unique $\mu_*\in [0,1-\|a\|_{\w}]$ such that $\mu_*=\|(A-\mu_* I)^{-1}b\|_1$. The proof is complete. 
 \end{proof}

\begin{remark}
Set $A=(2\gamma-\|a\|_{\w}) I-T_n(a)$ with $T_n(a)\ge 0$ and $\|a\|_{\w}<\gamma$. Suppose  $b\ge 0$ is such that $\|a\|_{\w }+\|b\|_1\le \gamma $ and $\|b\|_1\le (\gamma-\|a\|_{\w})^2$. If we replace $1-\|a\|_{\w}$ by $\gamma-\|a\|_{\w}$ in the above analysis, it can be proved that $g'(\mu)\le 0$, $g(0)>0$ and $g(\gamma-\|a\|_{\w})\le 0$, it follows that the function $g(\mu)$ has exactly one root in $[0,\gamma-\|a\|_{\w}]$.
 \end{remark}
 
{
We may extend the results in Theorem \ref{thm:solution} to a more general setting and obtain the following.
 \begin{theorem}\label{thm:solution2}
Suppose in equation \eqref{eq:equation1} the matrix $ A= (2\gamma -\|a\|_{\w })I-T_n(a)$ and  the nontrivial nonnegative vector $b$ satisfies that $\|b\|_1\le (\gamma-\|a\|_{\w})^2$, then among the solutions $x_*$ of equation \eqref{eq:equation1} there is exactly one which satisfy $\|a\|_{\w}+\|x_*\|_1\le \gamma $. 
 \end{theorem}
 }
 
 \subsection{Related rational functions}
As we can see, the functions $f(\mu)$ and $g(\mu)$ play important roles in the theoretical analysis of equation \eqref{eq:equation1}. In this subsection, we study some related functions that will be used in the subsequent analysis. For notational convenience, in what follows we set $\beta=1-\|a\|_{\w}$. 

Consider the function $g(\mu)=f(\mu)-\mu$, we have proved in the proof of Theorem  \ref{thm:solution} that $g'(\mu)<0$ for $\mu\in [0,\beta]$. Moreover, it holds that $g(\mu_*)=0$, where  $\mu_*=\|x_*\|_1\in [0,\beta]$ and $x_*$ is the unique solution of equation \eqref{eq:equation1} such that $\|x_*\|_1+\|a\|_{\w}<1$.  We may conclude the following  
\begin{lemma}\label{lem:g}
    The function $g(\mu)$ is monotonically decreasing in $[0,\beta]$, and it holds $g(\mu)>0$ for $\mu\in[0,\mu_*)$ and $g(\mu)<0$ for $\mu\in (\mu_*,\beta]$.
\end{lemma}

{Defining} $g_1(\mu)=\frac{\beta-\mu}{g(\mu)}$, $\mu\in [0,\mu_*)$, we have the following 
\begin{lemma}\label{lem:g1}
The function $g_1(\mu)$ is monotonically increasing on $[0,\mu_*)$.
\end{lemma}
\begin{proof}
    It suffices to prove $g_1'(\mu)>0$ for $\mu\in [0,\mu_*)$. A direct computation yields $g_1'(\mu)=\frac{\widehat{g}_1(\mu)}{g^2(\mu)}$, where 
 $\widehat{g}_1(\mu)= \beta - \|(A-\mu I)^{-1}b\|_1 - (\beta -\mu)\|(A-\mu I)^{-2}b\|_1.$ Recall that $\|b\|_1\leq (1-\|a\|_{\w})^2=\beta^2$, it follows that  $\|(A-\mu I)^{-1}b\|_1\leq \|(A-\mu I)^{-1}\|_1 \|b\|_1\leq \frac{\beta^2}{2\beta-\mu}$, we thus have
\[
\widehat{g}_1(\mu)\geq \beta-\frac{\beta^2}{2\beta-\mu}-\frac{\beta^2(\beta -\mu )}{(2\beta -\mu )^2}\\
=\frac{\beta(\beta-\mu)^2}{(2\beta -\mu )^2}.
\]
Since $\beta-\mu =1-\|a\|_{\w}-\mu\ne 0$ for $\mu \in [0,\mu_*)$, we have $2\beta -\mu \ne 0$, this implies $\widehat{g}_1(\mu)>  0$, and  {thus}  $g'_1(\mu)> 0$ for $\mu\in [0,\mu_*)$. 
\end{proof}

Another function that will be encountered is   $g_2(\mu) =- \frac{\mu}{g(\mu)}$, $\mu\in (\mu_*,\beta]$. {We} have 
\begin{lemma}\label{lem:g2}
The function $g_2(\mu)$ is monotonically decreasing on  $(\mu_*,\beta]$. 
\end{lemma}

\begin{proof}
We show $g_2'(\mu)\le 0$ for $\mu\in (\mu_*,\beta]$. A direct computation yields
	$$g'_2(\mu) = \frac{\mu\|(A-\mu I)^{-2}b\|_1 - \|(A-\mu I)^{-1}b\|_1 }{g^2(\mu)}.$$
Observe that 
	\begin{align*}
		 \mu \|(A-\mu I)^{-2}b\|_1 - \|(A-\mu I)^{-1}b\|_1
		&\le \|(A-\mu I)^{-1}b\|_1(\mu \|(A-\mu I)^{-1}\|_1 - 1)\\
		&\le  \|(A-\mu I)^{-1}b\|_1\Big(\frac{\mu}{2-2\|a\|_{\w}-\mu}-1\Big )\\
		&\leq 2\|(A-\mu I)^{-1}b\|_1\frac{(\mu+\|a\|_{\w }-1)}{2-2\|a\|_{\w}-\mu}\\
		&\le 0,
	\end{align*}
	where the last inequality follows from the fact that  $\mu \leq 1- \|a\|_{\w }$. Hence, $g'_2(\mu) \le 0$ for $\mu \in (\mu_*,\beta ]$. 
\end{proof}

Consider the  function 
\[
 {h_{\tau}(\mu)}=\tau \| (A-\mu I)^{-1}b\|_1+(1-\tau)\mu,\ \mu\in [0, \beta].
\]
 We study  properties of  $h_{\tau}(\mu)$ since this function will be used in the  analysis of a relaxed fixed-point iteration.   Observe that $\|(A-\beta I)^{-1}b\|_1\leq \frac{\|b\|_1}{1-\|a\|_{\w}}<1-\|a\|_{\w}$, from which we obtain $\beta-\|(A-\beta I)^{-1}b\|_1>0$, so that $\frac{\beta}{\beta-\|(A-\beta I)^{-1}b\|_1}$ is well-defined, and it is clear that  $\frac{\beta}{\beta-\|(A-\beta I)^{-1}b\|_1}>1$. 

\begin{lemma}\label{lem:closed}
For $0\leq \tau\leq \min\{\frac{\beta}{\|A^{-1}b\|_1}, \frac{\beta}{\beta-\|(A-\beta I)^{-1}b\|_1}\}$ it holds that  $ {h_{\tau}(\mu)}\in [0,\beta]$ for any $\mu \in [0,\beta]$. 
\end{lemma}

\begin{proof}
Observe that   ${h_{\tau}(\mu)}= \tau g(\mu ) + \mu$ and $h_{\tau}(\mu_*)=\tau g(\mu_*)+\mu_*=\mu_*\in [0,\beta]$. We next show that ${h_{\tau}(\mu)}\in [0,\beta]$ for $\mu\in [0,\mu_*)$ and $\mu\in (\mu_*,\beta]$, respectively.  

For $\mu\in [0,\mu_*)$,  we have from Lemma \ref{lem:g} that $g(\mu)>0$, so that ${h_{\tau}(\mu)} \ge \mu\geq  0$. It is left to show  ${h_{\tau}(\mu)}\le \beta$. In view of Lemma \ref{lem:g1}, we know that the function $g_1(\mu)=\frac{\beta-\mu}{g(\mu)}$ satisfies   $g_1(\mu) \ge g_1(0)=\frac{\beta }{\|A^{-1}b\|_1}$. Since $\tau \leq \frac{\beta }{\|A^{-1}b\|_1}$, it follows that  
\[
\tau \le  g_1(\mu)=\frac{1-\|a\|_{\w}-\mu}{g(\mu)},
\]
together with the fact that $g(\mu)>0$ for $\mu\in [0,\mu_*)$, we obtain 
\[
\tau g(\mu)+\mu \le 1-\|a\|_{\w},
\]
that is,  $ {h_{\tau}(\mu)}\leq \beta$ for $\mu \in [0,\mu_*)$.

Concerning the case where $\mu\in (\mu_{* }, \beta ]$, it follows from Lemma \ref{lem:g} that $g(\mu)\le 0$,   we obtain  $ h_{\tau}(\mu)\leq \mu  \leq \beta $. To  show  $h_{\tau}(\mu)\ge 0$,  
	consider the function $g_2(\mu) =- \frac{\mu}{g(\mu)}$. In view of Lemma \ref{lem:g2}, we have $g_2(\mu) \ge g_2(\beta )= \frac{\beta}{\beta - \|(A-\beta I)^{-1}b\|_1}\geq \tau$.  Together with the fact that $g(\mu)\leq 0$ for $\mu\in (\mu_*, \beta]$, we have 
	\[
	h_{\tau}(\mu)=\tau g(\mu)+\mu \geq g_2(\mu)g(\mu)+\mu =0.
	\] 
The proof is complete.
\end{proof}

\subsection{Related issues in graph theory}\label{lap}
The square root of Laplacian matrices in graph theory plays an important role in fractional diffusion on networks \cite{lap,lap2}. 
We consider the case where the Laplacian matrix is of the form $L=W-\one v^T$, where the matrix $W$ is a block upper triangular $M$-matrix and $v$ is a nontrivial nonnegative vector. For instance,  consider the undirected graph in Figure \ref{fig:sto2}, 
\begin{figure}
\center
  \includegraphics[width=0.6\textwidth]{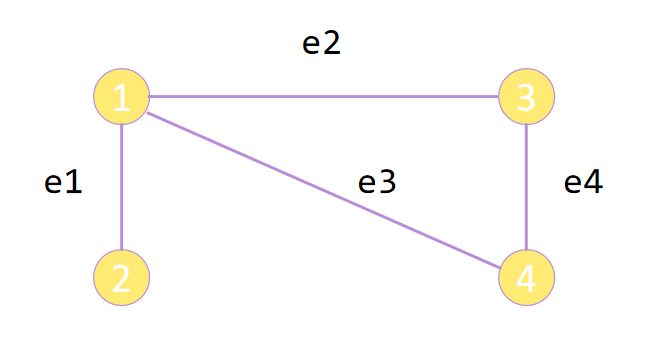}
\caption{An undirected graph.}\label{fig:sto2}     
\end{figure}
the Laplacian matrix is 
\[
L=\begin{pmatrix}
3& -1 & -1 & -1\\
-1 & 1 & 0 & 0\\
-1& 0 & 2 & -1\\
-1 & 0& -1 & 2
\end{pmatrix}=W-\one v^T,
\]
where $W=\begin{pmatrix}
4& -1 & -1 & -1\\
0 & 1 & 0 & 0\\
0& 0 & 2 & -1\\
0 & 0& -1 & 2
\end{pmatrix}$ is a block upper triangular $M$-matrix and $v=(1,0,0,0)^T$. 

Since the matrix $L$ here is a singular $M$-matrix, efficient methods to compute $L^{\frac12}$ would be Cyclic reduction (CR) or the Incremental
Newton (IN) iteration, both of which are  variants of the Newton's
method. However, the singularity of $L$ leads to linear convergence. In this section, we show that $L$ admits a unique square root $V-\one x^T$, where $x$ solves an equation in the format of \eqref{eq:equation1}. Applying the theoretical analysis for equation \eqref{eq:equation1},  we obtain that $L^{\frac12}=W^{\frac12}-\one v^TW^{-\frac12}$ (see the next Theorem \ref{thm:lap}).  Consequently, we can compute $W^{\frac12}$ and $W^{-\frac12}$ efficiently, and since $W$ is a nonsingular $M$-matrix,  quadratic convergence can be achieved in the iterative computation.

Observe that $W$ admits a unique square root that is also an $M$-matrix and $v$ is a nonnegative vector,  we show that  $L=W-\mathbf{1}v^T = (V - \mathbf{1}x^T)^2$, where  $V$ is the principal  square root of $W$ and $x$ is a nonnegative vector.  Without loss of generality, we write $L$ as $L=\ell (I-W_1-\one \tilde{v}^T)$. This way, we have $\tilde{v}=\frac{1}{\ell}v\ge 0$, and it follows from   $W=\ell(I-W_1) $ is an $M$-matrix that  $W_1\ge 0$ and $\rho(W_1)\le 1$. According to \cite[page 165]{higham}, we know that a natural iteration to compute the square root of $I-W_1-\one \tilde{v}^T$ is 
\begin{equation}\label{binomial}
X_{k+1}=\frac12(W_1+\one \tilde{v}^T+X_k^2),\quad X_0=0. 
\end{equation}
We have the following
\begin{lemma}\label{vge}
For a Laplacian matrix $L=\ell (I-W_1-\one \tilde{v}^T)$, if $W_1\geq 0$ and  $\tilde{v}\ge 0$, then the sequence $\{X_k\}$ generated by \eqref{binomial} satisfies $X_k=S_k+\one x_k^T$, where $S_k\geq 0$ and $S_k\one =s_k \one $  for some $s_k\ge 0$, $s_k\in \mathbb R$, and $x_k\geq 0$, $x_k\in \mathbb R^n$. 
\end{lemma}
\begin{proof}
 For $k=0$, we have $X_1=\frac12W_1+\frac12\one \tilde{v}^T$, from which we see that $S_1=\frac12W_1\geq 0$ and $x_1=\frac12 \tilde{v}\geq 0$.  Moreover, we have $S_1\one =\frac{1}{2}W_1\one$. Since $L\one =0$ and $\tilde{v}=\frac{1}{\ell}v\ge 0$, we obtain $W_1\one + \|\tilde{v}\|_1\one =\one $, and it follows from $W_1\ge 0$ that $W_1\one =(1-\|\tilde{v}\|_1)\one \geq 0$, so that $s_1=\frac12(1-\|\tilde{v}\|_1)\ge 0$. 
  
  Suppose $X_k=S_k+\one x_k^T$, where $S_k\geq 0$, $S_k\one =s_k\one$ for some $s_k\ge 0$, and $x_k\geq 0$, then we have
    \[
    \begin{aligned}
    X_{k+1}&=\frac12 (W_1+\one \tilde{v}^T+S_k^2+S_k\one x_k^T+\one x_k^TS_k+\one x_k^T\one x_k^T)\\
    &= \frac{1}{2}(W_1+S_k^2) +\frac12\one (\tilde{v}^T+s_kx_k^T+x_k^TS_k+\|x_k\|_1x_k^T),
   \end{aligned}
 \]
    from which one can see that $S_{k+1}=\frac12(W_1+S_k^2)\geq 0$ and $S_{k+1}\one =s_{k+1}\one $, where $s_{k+1}=\frac12(1-\|\tilde{v}\|_1+s_k^2)\ge 0$. Moreover, it can be seen that $x_{k+1}=\frac12((s_k+\|x_k\|_1)I+S_k^T)x_k+\frac12 \tilde{v}\geq 0$. 
\end{proof}

Observe that the sequence $\{S_k\}$ satisfies $S_{k+1}=\frac12(W_1+S_k^2), \ S_0=0$, according to \cite[page 165]{higham}, we know that $\{S_k\}$ converges to $V_1$ such that $(I-V_1)^2=I-W_1$.  On the other hand, according to the convergence of the sequence $S_k+\one x_k^T$, it can be seen that there is $\tilde{x}\geq 0$ such that $x_k$ converges to $\tilde{x}$. Hence, we have  $I-W_1-\one \tilde{v}^T=(I-V_1-\one \tilde{x}^T)^2$,  from which we know that $W-\one v^T=(V-\one x^T)^2$, where $V$ is the square root of $W$ and $x\geq 0$. Then we have   
\begin{align}\label{eq:lap0}
	W-\mathbf{1}v^T = V^2 - V\mathbf{1}x^T - \mathbf{1}x^TV + \mathbf{1}x^T\mathbf{1}x^T. 
\end{align}

Since for a Laplacian matrix $L=W-\one v^T$, we have $L\bf 1=0$, it follows that  $W {\bf 1} = \|v\|_1 \mathbf{1}$ if the vector $v$ is nonnegative. It can be verified analogously to \cite[Theorem 3.9]{arxiv_jie} that  $(V-\one x^T)\one =0$ and $V\mathbf{1} = \nu \mathbf{1}$, where  $\nu = \sqrt{\|v\|_1}$. Then equation \eqref{eq:lap0} is equivalent to 
\begin{equation}\label{eq:lap}
    (\nu I + V^T)x - \|x\|_1x = v. 
\end{equation}
 According to Lemma \ref{vge}, we know that $x\geq 0$ if $v\geq 0$, then we have from  $(V-\one x^T)\one =0$ that $\|x\|_1=\nu $. If $W$ is invertible, it follows that $V$ is invertible, then we obtain from \eqref{eq:lap} that $x=(V^T)^{-1}v$. 
Analogously to  Lemma \ref{lem:mu0}, we deduce that 
\begin{theorem}\label{thm:lap}
For a  Laplacian matrix $L=W-\one v^T$, where $W$ is an invertible $M$-matrix and $v\geq 0$, then equation \eqref{eq:lap} has a solution $x_*=( V^T)^{-1}v$, where $V$ is the principal square root of $W$. Moreover, let $\nu=\sqrt{\|v\|_1}$, then $\|x_*\|_1=\nu$ and $\nu$ solves the  equation of $\mu$, which is 
 \begin{equation}\label{eq:lmu}
 \mu=\|((\nu -\mu) I + V^T)^{-1}v\|_1. 
 \end{equation}
\end{theorem}

The following lemma shows that $\nu$ is the only solution of  equation \eqref{eq:lmu} in the interval $[0,\nu]$, which, in turn, confirms that $V- \one v^TV^{-1}$ is the unique square root of $L$ that is also an $M$-matrix. 

\begin{lemma}\label{lem:unique}
If $W$ is an invertible $M$-matrix and $v\ge 0$ is nontrivial,   then equation \eqref{eq:lmu} has one and only one solution in the interval $[0,\nu]$. 
\end{lemma}
\begin{proof}
Define $f(\mu) = \|((\nu-\mu)I + V^T)^{-1}v\|_1 - \mu$ for $\mu\in [0,\nu]$. Since $W$ is an invertible $M$-matrix and $V$ is the principal square root of $W$, it can be seen that  $(\nu -\mu)I + V^T$ is an invertible $M$-matrix for $\mu\in [0,\nu]$ and $((\nu -\mu)I + V^T)^{-1} \ge 0$.  Recall that  $V\one =\nu  \one$, from which one can prove that  $\one ^T((\nu -\mu)I + V^T)^{-1}=\frac{1}{2 \nu  -\mu}\one^T$ and 
\begin{align}\label{eq:1}
	\|((\nu -\mu)I + V^T)^{-1}v\|_1 = \sum_{i = 1}^{n} \frac{v_i}{2\nu -\mu} = \frac{\nu^2}{2\nu -\mu}, 
\end{align}
where the  equality follows from the fact that $\sum_{i=1}^nv_i=\|v\|_1=\nu^2$. Then,  we have from \eqref{eq:1} that $f(\mu) = \frac{\nu^2}{2\nu-\mu} - \mu$, $\mu\in [0,\nu]$. Observe that   $\frac{\nu^2}{2\nu-\mu} - \mu + 2\nu\geq 2\sqrt{\nu^2}=2\nu$, and the equality holds if $\mu = \nu$. It follows that  $f(\mu) = \frac{\nu^2}{2\nu-\mu} - \mu\geq 0$ for $\mu\in[0,\nu]$ and  $f(\mu)=0$ if and only if $\mu=\nu$. Hence, the function $f(\mu)$ has a unique root $\nu$ in the interval $[0,\nu]$.
\end{proof}

In view of  Theorem \ref{thm:lap} and Lemma  \ref{lem:unique}, we know that if a Laplacian matrix $L=W-\one v^T$ where $W=\begin{pmatrix} W_{1}& W_2\\O& W_3\end{pmatrix}$ is an invertible block upper triangular $M$-matrix, then  $L^{\frac12}=W^{\frac12}-\one v^TW^{-\frac12}$. Hence, instead of computing the square root of $L$, we may compute the square root of $W$. This way, the iteration of Denman and Beavers (DB) \cite{DB,higham}, defined as 
\begin{align*}
X_{k+1}&=\frac12 (X_k+Y_k^{-1}),\quad X_0=W,\\
Y_{k+1}&=\frac12 (Y_k+X_k^{-1}), \quad Y_0=I,
\end{align*}
serves as a good choice as it simultaneously yields  $W^{\frac12}$ and $W^{-\frac12}$, with  $\lim_{k\rightarrow \infty}X_k=W^{\frac12}$ and $\lim_{k\rightarrow \infty}Y_k=W^{-\frac12}$. 

\section{Numerical treatment}\label{comp}

The generalized equation in Theorem \ref{thm:solution2} can be reduced to the canonical form \eqref{eq:equation1} via a linear scaling by introducing the transformed variables $\tilde{x} = x/\beta$, $\tilde{A} = A/\beta$, and $\tilde{b} = b/\beta^2$. This transformation preserves the structural properties of \eqref{eq:equation1}, with the exception of the boundary case $\|b\|_1 = (1 - \|a\|_w)^2$, which corresponds to a numerical singularity where the iterative operator may lose its strict contraction property. To maintain conciseness and remain consistent with the motivating application of quasi-Toeplitz $M$-matrix square roots, our subsequent algorithmic analysis focuses on the standard form \eqref{eq:equation1} under the strict inequality in \eqref{eq:ba}. The treatment of the critical case, which involves more delicate convergence analysis, is deferred to future work.

 We  design and analyze the convergence of a relaxed fixed-point iteration and Newton's iteration, based on the  rational function $f(\mu)=\|(A-\mu I)^{-1}b\|_1$. Then, we show that a structure-preserving doubling algorithm can be applied,  prove that the sequences obtained by the doubling algorithm are well-defined and that the convergence is at least linear  with rate 1/2.

\subsection{A relaxed fixed-point iteration}

In \cite{arxiv_jie}, the fixed-point iteration 
\begin{equation} \label{fixed1}
	x_{k+1} = (A-\mu_k I)^{-1}b, \quad x_0 = 0,
\end{equation}
where $\mu_k= \|x_k\|_1$,  has been proposed for solving equation \eqref{eq:equation1} in the case where the coefficient matrices and vectors are of infinite sizes.  For the numerical treatment of equation \eqref{eq:equation1}, we consider a relaxed  version of  \eqref{fixed1}, that is, 
\begin{align}\label{para2}
	\begin{cases}
		x_{k+1} = (A-\mu_k I)^{-1}b ,\\
		\mu_{k+1} = \tau  \|x_{k+1}\|_1+ (1-\tau)\mu_k,
	\end{cases}
\end{align}
where $\tau\in (0,1]$ and  $\mu_0\in [0,\beta]$.  When $\tau=1$ and we choose $\mu_0=0$, it can be seen that the sequence $\{x_k\}$ consists in those defined in \eqref{fixed1}. 
To show the convergence of the sequence $\{x_k\}$ to the solution $x_*$ of equation \eqref{eq:equation1}, we have the following 
\begin{lemma}\label{mu_and_x}
If the sequence $\{\mu_k\}$ converges to a constant $\mu_*\in [0,\beta]$, then the sequence $\{x_k\}$ generated by \eqref{para2},  converges to the unique solution $x_*$ of equation \eqref{eq:equation1}, and   $x_*=(A-\mu_* I)^{-1}b$.
\end{lemma}

\begin{proof}
If $\lim_{k\rightarrow\infty}\mu_k=\mu_*$, then we can see from \eqref{para2} that $\mu_*$ satisfies 
\[
\mu_*=\tau \| (A-\mu_* I)^{-1}b\|_1+(1-\tau)\mu_*,
\]
which yields
\[
\mu_*=\| (A-\mu_* I)^{-1}b\|_1.
\]
Then, in view of Lemma \ref{lem:mu0} and Theorem \ref{thm:solution},  we know that $(A-\mu_* I)^{-1}b$ is the unique solution of equation \eqref{eq:equation1}. On the other hand, we know from \eqref{para2} that $\lim_{k\rightarrow \infty}x_k=(A-\mu_* I)^{-1}b$, this means the sequence $\{x_k\}$ converges to the required solution of equation \eqref{eq:equation1}. 
\end{proof}

According to Lemma \ref{mu_and_x}, one can see that in order to prove the convergence of the sequence $\{x_k\}$, it is sufficient to prove the convergence of the sequence $\{\mu_k\}$. We have the following convergence result.   

\begin{lemma}\label{lem:mu}
		Suppose $0<\tau 
		\leq \min\{2, \frac{\beta}{\|A^{-1}b\|_1}, \frac{\beta}{\beta-\|(A-\beta I)^{-1}b\|_1}\}$, then the sequence $\{\mu_k\}$ generated by \eqref{para2} converges to a constant $\mu_*\in [0,\beta]$.  
\end{lemma}

\begin{proof}
Consider the function $h_{\tau}(\mu)= \tau \|(A-\mu I)^{-1}b\|_1 + (1-\tau)\mu$, $\mu\in [0,\beta]$. According to Lemma \ref{lem:closed}, we know that $ h_{\tau}$  maps the closed interval $[0,\beta]$ into itself.
Observe that $h_{\tau}(\mu)=\tau f(\mu)+(1-\tau)\mu$, where $f(\mu)$ is defined in \eqref{eq:f}. In view of \eqref{eq:deri}, we obtain
$$
h'_{\tau}(\mu)=\tau(\|(A-\mu I)^{-2}b\|_1-1)+1.
$$
Recall from \eqref{ine:fu} that
$\|(A-\mu I)^{-2}b\|_1<1$,
 it follows 
 \begin{equation}\label{in:ge1}
h'_{\tau}(\mu)<1.
\end{equation}
On the other hand, 
it follows from  $(A-\mu I)^{-1}\geq A^{-1}\geq 0$ that
	\begin{align}\label{eq:tau}
		h'_{\tau}(\mu) 
				&\ge \tau (\|A^{-2}b\|_1-1) + 1. 
	\end{align}
Since  $0<\|A^{-2}b\|_1<\frac{\|b\|_1}{(1-\|a\|_{\w})^2}<1$ and  $\frac{2}{1-\|A^{-2}b\|_1}>2$ , it follows from $\tau\leq 2$ that $\tau (\|A^{-2}b\|_1-1)> -2$, and according to \eqref{eq:tau}, we obtain $ h'_{\tau}(\mu)>-1$.  Together with \eqref{in:ge1}, we obtain $|h'_{\tau}(\mu)|<1$ for $\mu\in [0,\beta]$. This implies that $ h_{\tau}$ is a contraction   on $[0,\beta]$. 	Hence, according to  the  Banach fixed-point theorem, there is $\mu_*\in [0,\beta]$ such that $\mu_*= h_{\tau}(\mu_*)$ and  any sequence $\{\mu_k\}$ generated by \eqref{para2} converges to $\mu_*$.
\end{proof}

It is natural to ask whether one can choose an optimal parameter 
$\tau$ to accelerate the convergence of iteration \eqref{para2}.
	We provide some insights on how to choose a proper parameter $\tau$. For any initial value $\mu_0\in [0,\beta]$, we have  $$\mu_*= h_{\tau}(\mu_0)+ h'_{\tau}(\mu_0)(\mu_*-\mu_0)+O(|\mu_*-\mu_0|^2),$$ 
    from which we obtain
    $$|\mu_*-\mu_1|=| h'_{\tau}(\mu_0)||\mu_*-\mu_0|+O(|\mu_*-\mu_0|^2).$$ 
 We may choose $\tau$ such that  $h'_{\tau}(\mu_0) = \tau (\|(A-\mu_0  I)^{-2}b\|_1-1) + 1=0$. This way, we obtain 
\begin{equation}\label{eq:tao0}
\tau(\mu_0) =\frac{1}{1-\|(A-\mu_0  I)^{-2}b\|_1}. 
\end{equation}
Recall that  to prove the convergence of iteration \eqref{para2},  Lemma \ref{lem:mu} requires that  
$$\tau\leq \min\{2, \frac{\beta}{\|A^{-1}b\|_1}, \frac{\beta}{\beta-\|(A-\beta I)^{-1}b\|_1}\}.$$
Therefore, for different choice of $\mu_0$, we  need to check if  $\tau(\mu_0)$ defined in \eqref{eq:tao0} satisfies this condition. If $\mu_0=0$, we have $\tau(0)=\frac{1}{1-\|A^{-2}b\|_1}$. We prove that this choice makes sense. Indeed, since $\|b\|_1 < (1-\|a\|_{\w})^2$ and $\|A^{-2}b\|_1 \leq  \frac{\|b\|_1}{(2-2\|a\|_{\w})^2} < \frac{1}{2}$,  where the first inequality for $\|A^{-2}b\|_1$ follows from $\|A^{-2}b\|_1\leq \|(A-\mu I)^{-2}b\|_1$ and \eqref{ine:fu}, we  know that
\begin{equation}\label{taole2}
  \tau(0)=  \frac{1}{1-\|A^{-2}b\|_1} < 2.
\end{equation}

On the other hand, recall that  $g'(\mu) = \|(A-\mu I)^{-2}b\|_1 - 1<0$ for $\mu\in [0,\beta]$, thus there exist $\xi \in (0,\beta)$ such that $\frac{g(0)-g(\beta)}{\beta} = -g'(\xi)$, that is, 
\begin{align}\label{eq:xi}
    \frac{\|A^{-1}b\|_1 + \beta-\|(A-\beta I)^{-1}b\|_1}{\beta} = -g'(\xi).  
\end{align}

Observe that  $(A-\xi I)^{-1}\geq A^{-1}\geq 0$ so that $(A-\xi I)^{-2}b\geq A^{-2}b\geq 0$,  it follows that $\|(A-\xi I)^{-2}b\|_1\geq \|A^{-2}b\|_1$, from which we obtain $g'(\xi)=\|(A-\xi I)^{-2}b\|_1-1\geq g'(0)=\|A^{-2}b\|_1-1$. Then, together with \eqref{eq:xi}, we have 
  $$1-\|A^{-2}b\|_1\geq \frac{\|A^{-1}b\|_1 + \beta-\|(A-\beta I)^{-1}b\|_1}{\beta}, $$ 
  which, together with \eqref{taole2}, implies 
  $$\tau(0)=\frac{1}{1-\|A^{-2}b\|_1}\le \min\{2, \frac{\beta}{\|A^{-1}b\|_1}, \frac{\beta}{\beta-\|(A-\beta I)^{-1}b\|_1}\}.$$

\subsection{Newton's iteration}

In this section, we show that the Newton's iteration can be applied to compute the solution of equation \eqref{eq:equation1}. In particular, we show that  the Newton's iteration can be used to find a solution of equation $g(\mu)=f(\mu)-\mu=0$, where $f(\mu )=\|(A-\mu I)^{-1}b\|_1$.

 Observe that when applying  Newton's iteration to equation $g(\mu) = 0$, we obtain  the sequence $\{\mu_k\}$ generated by
\begin{equation}\label{2.2}
	\mu_{k+1} = \mu_k - \frac{g(\mu_k)}{g'(\mu_k)} ,\quad \mu_0=0.
\end{equation}
Recall from \eqref{gu} that $g'(\mu)=\|(A-\mu I)^{-2}b\|_1-1<0$  for any $\mu \in [0,\beta]$,  
hence, to show  that the sequence $\{\mu_k\}$ is well-defined, it suffices to show $\mu_k\in [0, \beta ]$ for $k=0,1,2,\ldots.$ Since $\mu_*<\beta $, it suffices to show $0 \le \mu_k \le \mu_{k+1} \le \mu_*$. To this end, we need the following analysis. 

If $\mu_k\leq \mu_*$,  it follows from $0\le (A-\mu_k I)^{-1}\le (A-\mu_* I)^{-1}$ that
	\begin{align}
	&\|(A-\mu_k I)^{-1}b\|_1 - \|(A-\mu_* I)^{-1}b\|_1 \notag \\
		&= \one^{T} (A-\mu_k I)^{-1}[(A-\mu_* I) - (A-\mu_k I)](A-\mu_* I)^{-1}b  \nonumber \\
        &= {\|(A-\mu_k I)^{-1}(A-\mu_* I)^{-1}b\|_1(\mu_k - \mu_*)}, \notag 
\end{align}
together with the fact that $\mu_*=\|(A-\mu_* I)^{-1}b\|_1$, we obtain 
	\begin{align}\label{eq:muk}
	g(\mu_k)&=
		\|(A-\mu_k I)^{-1}b\|_1 -  \mu_k\notag  \\
		 &= \|(A-\mu_k I)^{-1}b\|_1 - \|(A-\mu_* I)^{-1}b\|_1 - \mu_k +\mu_*  \nonumber \\
        &={\big (\|(A-\mu_k I)^{-1}(A-\mu_* I)^{-1}b\|_1-1\big)(\mu_k - \mu_*). }
	\end{align}

Now we are ready to show the convergence of the sequence $\{\mu_k\}$ generated by Newton's iteration. We have
\begin{theorem} \label{Th2}
Suppose $x_*$   is the unique solution of equation \eqref{eq:equation1}  with $\|x_*\|_1<1-\|a\|_{\w}$, then  
 	the sequence $\{\mu_k\}$ generated by \eqref{2.2}  converges to $\mu_*$, where $\mu_*=\|x_*\|_1$, and it satisfies $0 \le \mu_k \le \mu_{k+1} \le \mu_*$. 
\end{theorem}

\begin{proof}
		For $k=0$, we have $\mu_1=-\frac{g(\mu_0)}{g'(\mu_0)}$, where $g(\mu_0)=\|A^{-1}b\|_1\geq 0$ and $g'(\mu_0)=f'(0)-1<0$ in view of \eqref{gu}, so that $\mu_1\geq \mu_0=0$.
 It follows from \eqref{eq:muk} that {$g(\mu_0)=\big (1-\|A^{-1}(A-\mu_* I)^{-1}b\|_1\big)\mu_*$}, from which we obtain  
	\begin{equation}\label{eq:uu}
	\begin{aligned}
    {\mu_*-\mu_1
	=\left(\frac{\|A^{-1}(A-\mu_* I)^{-1}b\|_1-\|A^{-2}b\|_1}{1-\|A^{-2}b\|_1}\right)\mu_*}.
	\end{aligned}
	\end{equation}

Since $0\leq A^{-1}\leq (A-\mu I)^{-1}$ for $\mu\in [0,\beta]$, it follows that  $0\leq A^{-2}\leq A^{-1}(A-\mu_* I)^{-1}$, so that $\|A^{-1}(A-\mu_* I)^{-1}b\|_1-\|A^{-2}b\|_1\geq 0$.   Together with \eqref{eq:uu}, we obtain $\mu_*-\mu_1\geq 0$, that is, $\mu_*\geq \mu_1$.

Now assume $0\leq \mu_k\leq \mu_{k+1}\leq \mu_*$.  Then  according to Lemma \ref{lem:g} and inequality \eqref{gu} we have  $g(\mu_{k+1})\ge 0$ and $g'(\mu_{k+1})<0$, from which we obtain 
\[
\mu_{k+2}-\mu_{k+1}=-\frac{g(\mu_{k+1})}{g'(\mu_{k+1})}\ge 0,
\]
  that is,  $\mu_{k+2}\geq \mu_{k+1}.$ It remains to show $\mu_{k+2}\leq \mu_*$. According to \eqref{eq:muk}, we have
\[
\begin{aligned}
\mu_*-\mu_{k+2}
&=c_{k+1}(\mu_*-\mu_{k+1}),
\end{aligned}
\]
where {$c_{k+1}=\left(\frac{\|(A-\mu_{k+1} I)^{-1}(A-\mu_* I)^{-1} b\|_1-\|(A-\mu_{k+1}I)^{-2}b\|_1}{1-\|(A-\mu_{k+1}I)^{-2}b\|_1}\right )$}.

Observe that $0\leq (A-\mu_{k+1} I)^{-1}\le (A-\mu_* I)^{-1}$, so that 
$$\|(A-\mu_{k+1} I)^{-1}(A-\mu_* I)^{-1} b\|_1\geq \|(A-\mu_{k+1} I)^{-2} b\|_1,$$ which implies that $c_{k+1}>0$. Hence $\mu_*-\mu_{k+2}\geq 0$, that is, $\mu_*\geq \mu_{k+2}$. Since the sequence $\{\mu_k\}$ is monotonically increasing and bounded above, there is $\mu\in [0,\mu_*]$ such that $\lim_{k\rightarrow \infty}\mu_k=\mu$ and $\mu=f(\mu)$. According to  Lemma \ref{lem:mu0} and Theorem \ref{thm:solution}, we know such $\mu$ is unique and $\mu=\mu_*$, that is, the sequence $\{\mu_k\}$ converges to $\mu_*$. The proof is complete. 
\end{proof}


\subsection{Doubling algorithm}
In this section, we show that a structure-preserving doubling algorithm can be applied to compute the solution $x_*$ of equation \eqref{eq:equation1}. To this end, we first recall the structure-preserving doubling algorithm (SDA) in \cite{dario_book}. 

The doubling algorithm can be used in computing $X{\in \mathbb R^{n\times n}}$ that satisfies the following system 
\begin{align}\label{DA}
	N \left(\begin{array}{cc}I\\X\end{array}\right) = K\left(\begin{array}{cc}I\\X\end{array}\right)W,
\end{align}
 where $N, K\in \mathbb R^{(m + n) \times (m + n)}$ and $W\in \mathbb R^{n\times n}$. If there is  $U$ and $L$ such that $LK = UN$, then it follows from \cite[page 147]{dario_book} that 
  \begin{align}\label{N1K1}
	LN \left(\begin{array}{cc}I\\X\end{array}\right) = UK\left(\begin{array}{cc}I\\X\end{array}\right)W^2. 
	\end{align}

 Set $N_1 = LN$, $K_1 = UK$, if for each $k$, there is $L_k$ and $U_k$ such that $L_kK_k=U_kN_k$, then by setting $N_{k+1} = L_kN_k$ and  $K_{k+1} = U_kK_k$, 
 a similar analysis as \eqref{N1K1} yields 
 \begin{align}\label{W^2^k}
	N_k \left(\begin{array}{cc}I\\X\end{array}\right) =  K_k\left(\begin{array}{cc}I\\X\end{array}\right)W^{2^k}.  
\end{align}
If the matrices $N$ and $K$ in \eqref{DA} are of the  form 
\begin{align}
	N =  \left(\begin{array}{cc}E &0 \\-P &I\end{array}\right),   K = \left(\begin{array}{cc}I &-Q\\0 &F\end{array}\right), 
\end{align} 
where $E, P, F, Q  \in \mathbb{R}^{n \times n}$, then  $N$ and $K$ are said to be in the standard structured form-I (SSF-I)  \cite[page 148]{dario_book}. Set $N_0 = N$ and $K_0 = K$. {If} for each $k\geq 0$ there are
\begin{align}\label{LU}
	L_k =  \left(\begin{array}{cc}L_{11}^{(k)} &0 \\L_{21}^{(k)} &I\end{array}\right),\quad    U_k = \left(\begin{array}{cc}I &U_{12}^{(k)}\\0 &U_{22}^{(k)}\end{array}\right), 
\end{align} 
  such that $L_kK_k = U_kN_k$, {then} for $k\geq 0$, the matrices  $K_{k+1}:=U_kK_k$ and $N_{k+1}:=L_kN_k$ preserve the standard structured form-I \cite{dario_book}, that is, 
 \begin{align}\label{NK}
 	N_{k} =  \left(\begin{array}{cc}E_{k} &0 \\-P_{k} &I\end{array}\right), \quad   K_{k} = \left(\begin{array}{cc}I &-Q_{k}\\0 &F_{k}\end{array}\right). 
 \end{align} 
 
 The following theorem provides sufficient conditions under which the matrices $L_k$ and $U_k$ in \eqref{LU} exist, so that the matrices $N_k$ and $K_k$ in \eqref{NK} are well defined. 
 \begin{lemma}\cite[Theorem 5.2]{dario_book}\label{lem:dabook}
	Let $k \ge 0$ and assume that the matrices $K_k$, $N_k$ are in the form
	\eqref{NK}. If the matrix $I - Q_k P_k$ is nonsingular, then also $I- P_k Q_k$ is nonsingular and	there exist matrices $L_k$, $U_k$ having the structure \eqref{LU} which satisfy $L_kK_k = U_kN_k$. Therefore, the matrices $K_{k+1} = U_kK_k$, $N_{k+1} = L_kN_k$ are well defined and have the
	structure \eqref{NK}, with
	\begin{equation}\label{EPFQ}
		\begin{aligned}
			E_{k+1}&=E_k(I-Q_kP_k)^{-1}E_k,\\
			 P_{k+1}&=P_k+F_k(I-P_kQ_k)^{-1}P_kE_k,\\
			F_{k+1}&=F_{k}(I-P_{k}Q_{k})^{-1}F_{k},\\
			Q_{k+1}&=Q_{k}+E_{k}(I-Q_{k}P_{k})^{-1}Q_{k}F_{k},
		\end{aligned}
	\end{equation}
where $E_0=E$, $P_0=P$, $Q_0=Q$ and $F_0=F$. 
 \end{lemma}

Now we are ready to show how a structure-preserving doubling algorithm can be applied to solve equation \eqref{eq:equation1}. 
 Suppose ${x_*}$ is a solution of equation \eqref{eq:equation1} such that ${x_*}\geq 0$ and $\|{x_*}\|_1+\|a\|_{\w}<1$.  Then
equation \eqref{eq:equation1} can be equivalently written as  
\begin{align}
	\left(\begin{array}{cc}0&\textbf{1}^{T}\\ -b&A\end{array}\right)\left(\begin{array}{cc}1\\ {x_*}\end{array}\right) = \left(\begin{array}{cc}1&0\\ 0&I\end{array}\right)\left(\begin{array}{cc}1\\ {x_*}\end{array}\right)\mu_*. \nonumber
\end{align}
Multiplying on the left by $\begin{pmatrix} 1& -\one^T A^{-1}\\ O& A^{-1} \end{pmatrix}$, we obtain 
 \begin{align}\label{SSF-I}
 	\left(\begin{array}{cc}\textbf{1}^{T}A^{-1}b&0\\ -A^{-1}b&I\end{array}\right)\left(\begin{array}{cc}1\\ {x_*}\end{array}\right) = \left(\begin{array}{cc}1&-\textbf{1}^{T}A^{-1}\\ 0&A^{-1}\end{array}\right)\left(\begin{array}{cc}1\\ {x_*}\end{array}\right)\mu_*. 
 \end{align}
 
It is clear that the matrices \[
N=\left(\begin{array}{cc}\textbf{1}^{T}A^{-1}b&0\\ -A^{-1}b&I\end{array}\right)\quad  {\rm and}\ \quad  K=\left(\begin{array}{cc}1&-\textbf{1}^{T}A^{-1}\\ 0&A^{-1}\end{array}\right)
\]
are in the form of SSF-I. 
For notational convenience, we replace $(E_k, P_k, Q_k, F_k)$ by $(c_k, u_k, v_k^T, F_k)$, and rewrite \eqref{EPFQ} as
\begin{equation}\label{EPFQ2}
		\begin{aligned}
			c_{k+1}&=c_k(1-v^T_ku_k)^{-1}c_k,\\
			u_{k+1}&=u_k+c_kF_k(I-u_kv^T_k)^{-1}u_k,\\
			v_{k+1}&=v_{k}+c_{k}(1-v^T_{k}u_{k})^{-1}F_{k}^Tv_{k},\\
			F_{k+1}&=F_{k}(I_{n}-u_{k}v^T_{k})^{-1}F_{k},
		\end{aligned}
	\end{equation}
where $c_0 = \textbf{1}^{T}A^{-1}b$, $u_0 = A^{-1}b$, $v_0 = A^{-T}\one $ and $F_0 = A^{-1}$. It can be seen that  for $k\ge 0$, $c_k \in \mathbb{R}$, $u_k \in \mathbb{R}^{n \times 1}$, $v_k \in \mathbb{R}^{n \times 1}$ and  $F_k \in \mathbb{R}^{n \times n}$.

   Concerning the computation of  $(I-u_kv^T_k)^{-1}$, which appears in the iterative process \eqref{EPFQ2}, since $(I-u_kv^T_k)^{-1}=I+ u_k(1-v_k^Tu_k)^{-1}v_k^T$, in the numerical treatment, it suffices to compute the scalar $(1-v_k^Tu_k)^{-1}$.

According to Lemma \ref{lem:dabook},  
we see that  if $1-v_k^Tu_k\ne 0$, then the sequences $\{c_{k+1}\}, \{u_{k+1}\}, \{v_{k+1}\}$ and $\{F_{k+1}\}$ are well-defined and  satisfy the iterative process  \eqref{EPFQ2}. In what follows,  we show that   $1-v_k^Tu_k\ne 0$ is always true for $k=0,1,2,\ldots$. To this end, we first investigate some properties of the sequences $\{c_k\}$, $\{u_k\}$, $\{v_k\}$ and $\{F_k\}$. 

\begin{lemma}\label{lem:abeta}
Set $\alpha_k = \|u_kv^T_k\|_1$, $\beta_k = \|c_kF_k\|_1$. 
If $\|u_kv_k^T\|_1<1$ for $k\in \mathbb Z$, then it holds 
\begin{equation}\label{eq:abeta1}
 \beta_{k+1} \leq \Big (\frac{\beta_k}{1-\alpha_k}\Big )^2, 
 \quad \alpha_{k+1} \le  \alpha_k\Big (1+\frac{\beta_k}{1-\alpha_k}\Big )^2.
\end{equation}
\end{lemma} 

\begin{proof}
Observe that 
 \begin{align}\label{betak1}
 \beta_{k+1} &= \|c_{k+1}F_{k+1}\|_1\nonumber \\
 & = \|(1-v_k^Tu_k)^{-1}c_kF_k(I-u_kv^T_k)^{-1}c_kF_k\|_1\nonumber\\
 & \le |(1-v_k^Tu_k)^{-1}|\beta_k^2\|(I-u_kv^T_k)^{-1}\|_1.
 \end{align}
 
If $\|u_kv_k^T\|_1<1$, we have
 \begin{equation}\label{eq:alpha11}
 \|(I-u_kv^T_k)^{-1}\|_1=\|\sum_{i=0}^{\infty}(u_kv_k^T)^i\|_1\leq \sum_{i=0}^{\infty}\alpha_k=\frac{1}{1-\alpha_k}.
 \end{equation}
 On the other hand,  one can verify  that 

  \begin{equation}\label{eq:uvk}
 \|u_kv_k^T\|_1=\|u_k\|_{1}\|v_k\|_{\infty}\geq |v^T_ku_k|,
 \end{equation}
 from which we obtain 
 \begin{equation}\label{eq:alpha2}
|(1-v_k^Tu_k)^{-1}|\leq \frac{1}{1-|v_k^Tu_k|}\leq \frac{1}{1-\alpha_k}. 
 \end{equation}
 It follows from \eqref{betak1}, \eqref{eq:alpha11} and \eqref{eq:alpha2} that
 \[
\beta_{k+1} \leq \big(\frac{\beta_k}{1-\alpha_k}\big)^2.  
 \]
 On the other hand, we have from \eqref{eq:alpha11} and \eqref{eq:alpha2} that 
  \begin{align*}
 \alpha_{k+1}&= \|u_{k+1}v^T_{k+1}\|_1 \\
 &= \|u_kv^T_k + (1-v^T_ku_k)^{-1}u_kv^T_kc_kF_k + c_kF_k(I-u_kv^T_k)^{-1}u_kv^T_k \\
 &+ c_kF_k(I-u_kv^T_k)^{-1}u_kv^T_k(1-v^T_ku_k)^{-1}c_kF_k\|_1\\ 
 &\le \alpha_k + 2\alpha_k\frac{\beta_k}{1-\alpha_k} + \alpha_k \big (\frac{\beta_k}{1-\alpha_k}\big)^2 \\
 &= \alpha_k\big(1+\frac{\beta_k}{1-\alpha_k}\big)^2. 
 \end{align*}
The proof is complete. 
\end{proof}

In what follows, we show that $\|u_kv^T_k\|_1<1$ is always true  for $k=0,1,2\ldots$, which, together with \eqref{eq:uvk}, implies that $1-v_k^Tu_k\ne 0$ for $k=0,1,2,\ldots$.

 \begin{lemma}\label{alphabeta}
 	Let $\alpha_k = \|u_kv^T_k\|_1$, $\beta_k = \|c_kF_k\|_1$, then it holds 
 	
 	 \begin{align}\label{eq:abeta}
 		\alpha_k <\Big (\frac{2^k}{2^k+1}\Big)^2	\ {\rm and}\  \beta_k < \frac{1}{(2^k+1)^2}.
 		\end{align}
\end{lemma}
 
 \begin{proof}
 Observe that $$\|A^{-1}\|^{2}_1\|b\|_1 \le \frac{\|b\|_1}{(2-2\|a\|_{\w})^2} < \frac{(1-\|a\|_{\w})^2}{(2-2\|a\|_{\w})^2}=\frac{1}{4},$$
from which  we get $\alpha_0 \le \|A^{-1}\|^{2}_1\|b\|_1 <\frac{1}{4}$ and   $\beta_0 \le \|A^{-1}\|^{2}_1\|b\|_1 <\frac{1}{4}$. Hence, \eqref{eq:abeta} holds true for $k=0$.
  
Assume  $\alpha_k <\big (\frac{2^k}{2^k+1}\big )^2$ and $\beta_k < \frac{1}{(2^k+1)^2}$, then according to \eqref{eq:abeta1}
we have 
 
\begin{align*}
    \alpha_{k+1} \le \alpha_k\Big (1+\frac{\beta_k}{1-\alpha_k}\Big )^2
                  = \Big (\frac{2^k}{2^k+1}\cdot\frac{2(2^k+1)}{1+2^{k+1}}\Big )^2= \Big (\frac{2^{k+1}}{2^{k+1}+1}\Big)^2,
\end{align*}
and 
\[
\beta_{k+1} \le \Big (\frac{\beta_k}{1-\alpha_k}\Big )^2 < \frac{1}{(2^{k+1}+1)^2}.
\] 
Hence, the inequalities in  \eqref{eq:abeta} are always satisfied for $k=0,1,\ldots$. 
\end{proof}

 According to  \eqref{eq:abeta}, we can see that $\|u_kv_k^T\|_1<1$ holds for $k=0, 1,2,\ldots$, implying that  $I - u_kv_k^T$ is nonsingular. It follows from \eqref{eq:uvk} that 
 \begin{equation}\label{eq:vu}
 1-v^T_ku_k\ne 0,
 \end{equation}
  so that  
the sequences $\{c_k\}, \{u_k\}, \{v_k\}$ and $\{F_k\}$ are well-defined. Moreover, we have 
 
 \begin{theorem}\label{thm:nonnegative}
 The sequences $\{c_k\}$, $\{u_k\}$, $\{v_k\}$ and $\{F_k\}$ generated by \eqref{EPFQ2}, with $c_0=\one^TA^{-1}b$, $u_0=A^{-1}b$, $v_0=A^{-T}\one$ and $F_0=A^{-1}$, are nonnegative sequences.
 \end{theorem}
 
 \begin{proof}
Recall that $A$ is a nonsingular $M$-matrix, so that $A^{-1}\geq 0$. Since $b$ is a nonnegative vector, it can be seen that $c_0, u_0, v_0$ and $F_0$ are nonnegative. Suppose  $\{c_k\}$, $\{u_k\}$, $\{v_k\}$ and $\{F_k\}$  are nonnegative sequences. Then
it follows from $\|u_kv_k^T\|_{\infty}<1$ that $I-u_kv_k^T$
is a nonsingular $M$-matrix, so that  $(I-u_kv_k^T)^{-1}\geq 0$, from which we obtain that $u_{k+1}$ and $F_{k+1}$ are nonnegative. 

On the other hand, we have from \eqref{eq:uvk} that $v^T_ku_k=|v_k^Tu_k|\leq \|u_kv_k^T\|_1<1$, so that $(1-v_k^Tu_k)^{-1}\geq 0$, from which it follows that $c_{k+1}\geq 0$ and $v_{k+1}\geq 0$. 
 \end{proof}

In view of Lemma \ref{lem:dabook} and  inequality \eqref{eq:vu}, we know that  
the structure-preserving doubling algorithm can be applied in computing $x_*$ that satisfies equation \eqref{SSF-I}, and we obtain 
\begin{align}\label{mu^*}
 			\left(\begin{array}{cc}c_k &0 \\-u_k &I\end{array}\right) \left(\begin{array}{cc}1\\x_*\end{array}\right) =  \left(\begin{array}{cc}1 &-v^T_k\\0 &F_k\end{array}\right)\left(\begin{array}{cc}1\\x_*\end{array}\right){\mu}_*^{2^k},  
 	\end{align}
 where  $\{c_k\}$, $\{u_k\}$, $\{v_k\}$ and $\{F_k\}$ are generated by \eqref{EPFQ2}.  We show in the following theorem  the convergence of the  sequence $\{u_k\}$  to $x_*$. 
 
 \begin{theorem}\label{thm:convergence}
 	Suppose $x_*$ is the unique solution of equation \eqref{eq:equation1} such that $\|x_*\|_1<1-\|a\|_{\w}$,  then the sequence $\{u_k\}$ generated by \eqref{EPFQ2} converges to $x_*$ and the convergence is at least R-linear with rate $1/2$.
 \end{theorem}
 
 \begin{proof}
 Observe that a direct computation of \eqref{mu^*} yields 
\begin{equation}\label{eq:ck}
	c_k=(1-v_k^T{x_*}){\mu}_*^{2^k}
\end{equation}
 	\begin{equation}\label{eq:pkx}
 		{x_*}-u_k=F_k{x_*}{\mu}_*^{2^k}. 
 	\end{equation}

We prove by induction that for $k=0,1,2,\ldots$
\begin{equation}\label{eq:vx}
\|{x_*}v_k^T\|_1\leq \frac{2^k}{1+2^k}.
\end{equation}

For $k=0$, we have 
\begin{align*}
\|{x_*}v_0^T\|_1=\|{x_*}\one ^TA^{-1}\|_1
\leq \|A^{-1}\|_1\|{x_*}\|_1
<\frac{1-\|a\|_{\w}}{2(1-\|a\|_{\w})}=\frac12,
\end{align*}
so that inequality \eqref{eq:vx} holds true for $k=0$. Suppose  	it holds for $k=i$, that is, $\|x_*v_i^T\|_1\leq \frac{2^i}{1+2^i}$,  then for $i+1$, we have
\[
\begin{aligned}
\|{x_*}v_{i+1}^T\|_1&=\|{x_*}(v_i^T+c_i(1-v_i^Tu_i)^{-1}v_i^TF_i)\|_1\\
&=\|{x_*}v_i^T+(1-v_i^Tu_i)^{-1}{x_*}v_i^T(c_iF_i)\|_1\\
&\leq \|{x_*}v_i^T\|_1+\frac{\beta_i}{1-\alpha_i}\|{x_*}v_i^T\|_1\\
&\leq \Big(1+\frac{1}{2^{i+1}+1 }\Big)\frac{2^i}{1+2^i}\\
&= \frac{2^{i+1}}{1+2^{i+1}},
\end{aligned}
\]
where the first inequality follows from \eqref{eq:uvk} and the second inequality follows from \eqref{eq:abeta}. Hence, $\|{x_*}v_k^T\|_1\leq \frac{2^k}{1+2^k}$ holds for $k=0,1,2,\ldots$. 

On the other hand,  analogously to \eqref{eq:uvk}, one can see that  $v_k^T{x_*}\leq |v_k^T{x_*}|\leq \|{x_*}v_k^T\|_1$, so that   $v_k^T{x_*}\leq \frac{2^k}{1+2^k}$. Then it follows from \eqref{eq:ck} that  ${\mu}_*^{2^k}\leq {c_k}(2^k+1)$, and together with \eqref{eq:abeta} and \eqref{eq:pkx}, we have
\begin{align}\label{eq:xuk}
\|{x_*}-u_k\|_1&\leq {\beta_k}(1+2^k)\|{x_*}\|_1\leq \frac{ \|x_*\|_1(1+2^k)}{(1+2^k)^2} < \frac{1}{2^k}.
\end{align}
Hence, we obtain $\lim_{k\rightarrow \infty}\sqrt[k]{\|{x_*}-u_k\|_1}\leq \frac{1}{2}$, that is, the sequence $\{u_k\}$ converges to ${\mu_*}$ at least  R-linearly with rate $1/2$. 
\end{proof}

 Theorem \ref{thm:convergence} shows that the structure-preserving doubling algorithm has a convergence which is at least R-linear with rate $1/2$. However, in the numerical experiments, we observe that the convergence is linear for the first few steps while a quadratic convergence occurs later.


\section{Numerical experiments}
In this section,  we show by numerical experiments  the effectiveness of the relaxed fixed-point iteration \eqref{para2}, Newton's iteration \eqref{2.2} and the structure-preserving doubling algorithm(SDA) \eqref{EPFQ2} in computing the required solution $x_*$  of equation  \eqref{eq:equation1}. The iteration is  terminated if 
\[\|Ax_k-\mu_kx_k-b\|_1 \le 10^{-15}.\]
The tests were performed in MATLAB/version R2024a on a PC with an Intel(R) Core(TM) i5-9300H CPU @ 2.40GHz processor and 8 GB main memory.

\begin{example}\label{ex1}

We establish a semi-infinite quasi-Toeplitz matrix  $T(a)$ with  $a(z) = \sum_{i = -p}^{q}a_iz^i$,  relying on the package CQT-Toolbox of \cite{BSR}, which is available at https://github.com/numpi/cqt-toolbox. The construction of $T(a) = \tt cqt(a_n,a_p)$ in MATLAB is as follows 
		\[{\tt a_1=rand(p,1);\  a_2=rand(q,1);}\]
		\[{\tt a_1(1) = a_2(1); s=sum(a_1)+sum(a_2);}\]
		\[{\tt a_n=a_1/s;\ 	a_p=a_2/s};\]
	
Then we compute the square root of $I-T(a):=T(1-a)$ by the evaluation{/}interpolation technique. Set $\hat{A} = (2-\|g\|_{\w})I - T(g)$, where $I-T(g)$ is the Toeplitz part of the square root of $I-T(a)$.  Let $\hat{b}$ be an infinite vector whose first $n$ elements are nonzero and it satisfies that $\|\hat{b}\|_1 <(1-\|g\|_{\w})^2$.

Set $p = 3$, $q = 10$, $n = 400$. For $m = 2000$,  we truncate the matrix $\hat{A}$ and vector $\hat{b}$ to an $m$-dimensional matrix $A$ and vector $b$.   For different parameters $\tau$, we apply the relaxed fixed-point iteration $\eqref{para2}$ to calculate the solution $x_*$ of the equation $\eqref{eq:equation1}$. 

\end{example}
	
According to \eqref{eq:tao0},  for different choices of the initial value $\mu_0\in [0,\beta]$, we set the parameter $\tau $ as $\tau(\mu_0)=\frac{1}{1-\|(A-\mu_0  I)^{-2}b\|_1}$.  The number of iterations, CPU time, and relative residuals are reported in Table \ref{table4}. For the case  where $\tau = 1$, the relaxed fixed-point iteration reduces to the fixed-point iteration \eqref{fixed1}. From Tables \ref{table4} and \ref{table5}, we see that when the starting point is set to $\mu_0=0$, the relaxed fixed-point iteration with $\tau=\tau(0)$ is more efficient.

\begin{table}
	\centering
	\begin{tabular}{cccc}
		\toprule
		$\tau$&Number of Iterations&CPU time(seconds) & Residuals \\
		\midrule
		1&18&1.5000&3.0272e-16\\
		$\tau(0)$&11&0.9038&3.1861e-16\\
        $\tau(\frac{\beta}{2})$ &12&1.0044&4.9872e-16\\
		$\tau(\beta)$&59&4.8394&9.7740e-16\\
		\bottomrule
	\end{tabular}
    \caption{\textbf{Comparison of  the relaxed fixed-point iteration \eqref{para2}:  the parameter $\tau$ is chosen according to the choice of the initial points $\mu_0$. } }\label{table4}
\end{table}

\begin{table}
	\centering
	\begin{tabular}{cccc}
		\toprule
		$\tau$&Number of Iterations&CPU time(seconds) & Residuals \\
		\midrule
		0.1&319&26.0767&9.2822e-16\\
		0.5&55&4.6749&9.9874e-16\\
            0.9&23&1.8630&3.3170e-16 \\
		$\tau(0)$&11&0.9038&3.1861e-16\\
            1.5&24&2.0857&7.1183e-16 \\
            1.9&61&4.8924&7.9216e-16 \\
		\bottomrule
	\end{tabular}
    \caption{\textbf{Comparison of  the relaxed fixed-point  iteration \eqref{para2}: the initial points is set to $\mu_0 = 0$, the parameter $\tau$ is chosen randomly. } }\label{table5}
\end{table}

\begin{figure}
\centerline{\includegraphics[width=0.45\textwidth]{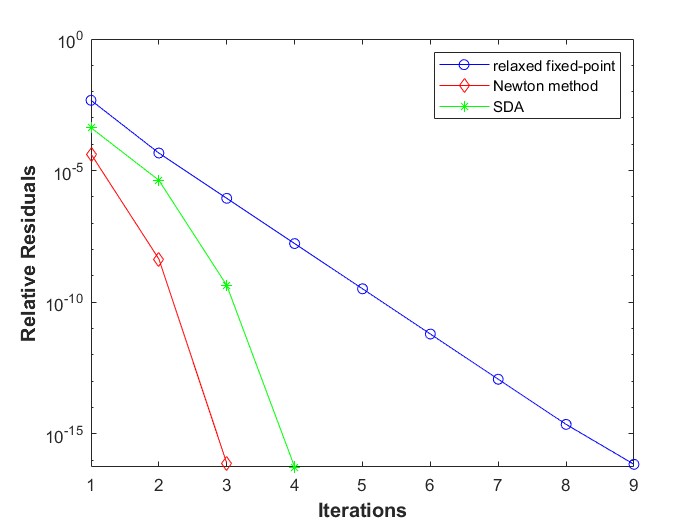}}
\caption{Comparison of Newton's iteration, SDA, and relaxed fixed-point iteration with parameter $\tau(0)$.}
\label{fig:comparision_FNS}
\end{figure}

In Figure \ref{fig:comparision_FNS}, we compare the number of iterations required by the relaxed fixed-point iteration, Newton's iteration and SDA. It can be seen that the number of iterations required by Newton's iteration and SDA is much less
than the number of iterations required by the relaxed  fixed-point iteration.

\begin{example}\label{ex2}
We consider a general $M$-matrix $A$, where $A=(2-\|M\|_{1})I - M$ with $M\geq 0$ and $\|M\|_1<1$. The matrix $M$ is constructed in MATLAB as 
		  \[{\tt M = {{rand(n,n)}}; \ M = M/(norm(M,1)+\delta*norm(M,1) ),}\]
    where $\delta>0$. 
 The vector $b$ is constructed as follow 
\[{\tt b=rand(n,1);\  b=(1-\|M\|_{\infty}-\sigma)^2*b/\|b\|_1. }\]
where $0<\sigma<1-\|M\|_{\infty}$. 
\end{example}

In this example, we set $n = 1000$, $\delta = 0.9$, $\sigma = 0.001$, and $n=100$, $\delta=0.01$, $\sigma=0.00001$, respectively. This way, the matrix 
$A$ is well-conditioned in the first case, but nearly singular in the second. We examine the effectiveness of the relaxed fixed-point iteration \eqref{para2}, Newton's iteration \eqref{2.2} and the structure-preserving doubling algorithm \eqref{EPFQ2} in computing the solution $x_*$ of equation \eqref{eq:equation1} such that $\|x_*\|_1+\|M\|_{1}<1$.

For the relaxed fixed-point iteration, we choose $\tau=\tau(0)=\frac{1}{1-\|A^{-2}b\|_1}$. The number of iterations, relative residuals and CPU time are reported in Table $\ref{table2}$ and Table $\ref{table3}$, respectively.

\begin{table}
	\centering
	\begin{tabular}{cccc}
		\toprule
		Algorithms &Number of Iterations&CPU time(seconds) & Residuals \\
		\midrule
		RFPI&{{51}} &{{0.9317}} & {{5.8414e-16}} \\
		NI&{{6}}&{{0.4963}} & {{2.2542e-16}} \\
        SDA&{{7}} &{{0.7221}}&{{2.2602e-16}} \\
		\bottomrule
	\end{tabular}\caption{\textbf{ Comparison of the relaxed fixed-point iteration (RFPI), Newton's iteration ( NI) and SDA, when $n = 1000$, $\delta = 0.9$, $\sigma = 0.001$.}}\label{table2}
\end{table}

\begin{table}
	\centering
	\begin{tabular}{cccc}
		\toprule
		Algorithms &Number of Iterations&CPU time(seconds) & Residuals \\
		\midrule
		RFPI &{{6}} &{{0.1088}} & {{6.2246e-17}} \\
		NI&{{3}} &{{0.3035}} & {{6.9468e-19}} \\
        SDA&{{4}} &{{0.3733}} & {{6.8145e-19}} \\
		\bottomrule
	\end{tabular}\caption{\textbf{ Comparison of the  relaxed fixed-point iteration (RFPI), Newton's iteration (NI) and SDA, when $n = {100}$, $\delta = {0.01}$, $\sigma = 1\times 10^{-5}$.}}\label{table3}
\end{table}

In the first case, Table \ref{table2} shows that  Newton's iteration is {superior} to both the SDA and the relaxed fixed-point iteration in terms of the number of iterations and CPU time.  In the second case, where the matrix $A$ is nearly singular, we see from Table \ref{table3} that the relaxed-fixed point iteration takes less time than Newton's iteration and SDA. This provides an example where the relaxed fixed-point iteration outperforms the other two methods in terms of efficiency.

\begin{figure}
	\center
\includegraphics[width=0.45\textwidth]{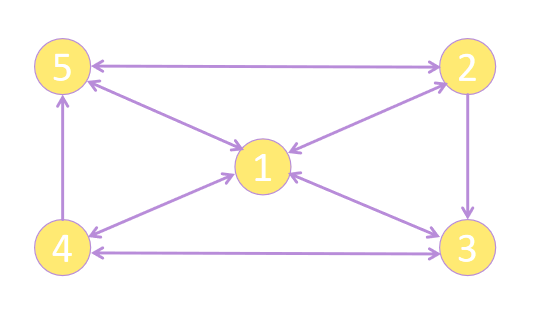}
	\caption{A directed graph with five nodes}\label{graph22}
\end{figure}

\begin{example}\label{L}
   The   Laplacian matrix $L$  of  the directed  graph from Figure \ref{graph22} is 
$$L=\left(\begin{array}{ccccc}4 & -1 & -1 & -1 & -1 \\ -1 & 3 & -1 & 0 & -1 \\ -1 & 0 & 2 & -1 & 0 \\ -1 & 0 & -1 & 3 & -1 \\ -1 & -1 & 0 & 0 & 2\end{array}\right).$$
It can be seen that $L=W-\one v^T$, where  $$W = \left(\begin{array}{ccccc}5 & -1 & -1 & -1 & -1 \\ 0 & 3 & -1 & 0 & -1 \\ 0 & 0 & 2 & -1 & 0 \\ 0 & 0 & -1 & 3 & -1 \\ 0 & -1 & 0 & 0 & 2\end{array}\right)$$
 and  the vector $v = [1,0,0,0,0]^T$. 
\end{example}

According to Section \ref{lap}, we know that $L^{\frac12}=W^{\frac{1}{2}}- \one y^T$, where $y=(W^{-\frac12})^Tv$. We compute $W^{\frac12}$ and  $W^{-\frac12}$ by the DB iteration and obtain 
$$y=[0.4772 \quad 0.1097\quad  0.1667\quad   0.1097\quad   0.1667]^T.$$

\begin{example}
Suppose $A=\begin{pmatrix}3& -1/4 &-1/4& -1/4& -1/4\\
     -1/4 &  3 & -1/4 & -1/4&  -1/4\\
     -1/4 & -1/4 & 3 & -1/4&  -1/4\\
    -1/4 & -1/4 & -1/4 &  3& -1/4\\
    -1/4& -1/4&  -1/4& -1/4 & 3
    \end{pmatrix}$ and $b=[1,0,0,0,0]^T$. We apply Newton's iteration and SDA to equation $Ax-\|x\|_1x=b$. 
\end{example}

According to Theorem \ref{thm:solution2}, this equation has a unique solution. Furthermore, it can be verified that $\mu_* = 1$ is a double root of the corresponding function $g(\mu)=0$. Because of this singularity, both methods drop to a linear convergence rate.  However,  a  difference in stability occurs:  the Newton's iteration achieves a relative error of only $1.7684 \times 10^{-8}$, whereas SDA successfully maintains a much higher precision, achieving a relative error of $9.1093 \times 10^{-14}$.

\section{Conclusions}

We investigated the theoretical and computational properties of a vector equation $Ax-\|x\|_1x=b$, where $A=(2-\|a\|_{\w})I-T_n(a)$ is an invertible $M$-matrix of Toeplitz structure, and $b$ is a nonnegative vector.  We showed that such equation has a unique solution $x_*$ such that $\|x_*\|_1+\|a\|_{\w}<1$ under the condition that $\|b\|_1<(1-\|a\|_{\w})^2$. Such solution is of interest when related to the square root of certain  invertible $M$-matrix with a rank-1 correction.  We proposed a relaxed fixed-point iteration and Newton's iteration for the computation of such solution. Moreover, we showed that a structure-preserving doubling algorithm can be applied in computing the solution, the convergence rate is at least 1/2. 


\section*{Financial Disclosure}

This work has been partially supported by the National Natural Science Foundation of China under grant No. 12201591 and
the National Research Foundation of Korea (NRF) grant funded by the Korea government (MOE) (RS-2023-00240475).

\section*{Conflict of interest}

The authors declare no potential conflict of interests.

\end{document}